\documentclass[11pt, twoside, leqno]{amsart} 
\usepackage{graphicx}
\usepackage{amsmath, amssymb, amsthm}
\usepackage[utf8]{inputenc}
\usepackage{hyperref}
\usepackage{fancyhdr}
\usepackage{comment}
\usepackage{placeins} 
\usepackage{comment}
\usepackage[numbers,sort&compress]{natbib}
\usepackage[font=small]{caption}
\usepackage{float}
\usepackage{empheq}
\usepackage{pgfplots}
\pgfplotsset{compat=1.14}
\usetikzlibrary{arrows.meta}
\usetikzlibrary{arrows,decorations.pathmorphing}
\usetikzlibrary{patterns,shapes.geometric}
\usepackage{calc}
\usepackage{stackengine,scalerel}
\usepackage{enumitem}%

\usepackage[makeroom]{cancel}

\newcommand{\Rs}{\mathbb{R}}

\newcommand{\I}{\mathcal{I}}
\newcommand{\dx}{\,\mathrm{d}x}
\newcommand{\ds}{\,\mathrm{d}\gamma}
\newcommand{\dt}{\,\mathrm{d}s}
\newcommand{\diff}{\mathrm{d}}
\newcommand{\dg}{\,\mathrm{d}\gamma}
\newcommand{\bigO}{\mathcal{O}}

\renewcommand{\div}{\operatorname{div}}
\newcommand{\pdfrac}[2]{\frac{\partial #1}{\partial #2}}
\newcommand{\diffrac}[2]{\frac{\diff #1}{\diff #2}}
\newcommand*\laplace{\mathop{}\!\mathbin\bigtriangleup}

\renewcommand*\laplace{\Delta}
\newcommand{\supp}{\operatorname{supp}}

\newcommand{\intt}{\int_{0}^{T}}
\newcommand{\ints}{\int_{0}^{t}}

\newcommand{\intg}{\int_{\bndry}}

\newcommand{\normtso}[3]{\|#1 \|_{L^#2(L^#3(\Omega))}}

\newcommand{\normso}[2]{\|#1 \|_{L^ #2(\Omega)}}

\newcommand{\z}{z^d}
\newcommand{\zd}{\dot z^d}

\newcommand{\nzd}{\nabla \dot z^d}
\newcommand{\nz}{\nabla z^d}
\newcommand{\psid}{\dot\psi}
\newcommand{\psidd}{\ddot\psi}
\newcommand{\npsi}{\nabla \psi}
\newcommand{\npsid}{\nabla \dot \psi}

\newcommand{\inv}[1]{\frac{1}{#1}}
\newcommand{\normho}[2]{\|#1\|_{H^ #2(\Omega)}}
\newcommand{\normtho}[3]{\|#1\|_{L^#2(H^{#3}(\Omega))}}

\newcommand{\normthg}[3]{\|#1\|_{H^#2(H^{#3}(\bndry))}}

\newcommand{\normtlhg}[3]{\|#1\|_{L^#2(H^{#3}(\bndry))}}
\newcommand{\normhg}[2]{\|#1\|_{H^{#2}(\bndry)}}
\newcommand{\normfspe}{\|f\|_{W^{1,1}(L^2(\Omega))}}
\newcommand{\ctr}{C_{\textup{tr},\bndry}}
\newcommand{\co}{C_{\Omega,4}}

\newcommand{\ce}{C_{E,\Omega}}

\newcommand{\normx}[1]{\|#1\|_{\hat{X}}}
\newcommand{\tpsi}{\tilde{\psi}}
\newcommand{\bndry}{\partial \Omega}
\newcommand{\tp}{\tilde{p}}
\newcommand{\dxs}{\, \textup{d}x\textup{d}s}
\newcommand{\ula}{\underline{a}}
\newcommand{\ola}{\overline{a}}

\usepackage{calc}
\usepackage{stackengine,scalerel}
\newcommand\dhookrightarrow{\mathrel{\ThisStyle{\abovebaseline[-.6\LMex]{%
  \ensurestackMath{\stackanchor[.15\LMex]{\SavedStyle\hookrightarrow}{%
  \SavedStyle\hookrightarrow}}}}}}

\usepackage{accents}
\newcommand{\dbtilde}[1]{\accentset{\approx}{#1}}

\def\env@cases#1{%
  \let\@ifnextchar\new@ifnextchar
  \left\lbrace\def\arraystretch{1.2}%
  \array{@{}#1@{\quad}l@{}}}
\makeatother

\newcommand\smallO{
  \mathchoice
    {{\scriptstyle\mathcal{O}}}
    {{\scriptstyle\mathcal{O}}}
    {{\scriptscriptstyle\mathcal{O}}}
    {\scalebox{.6}{$\scriptscriptstyle\mathcal{O}$}}
  }
\usepackage{enumitem}%

\usepackage{cleveref} 
\crefname{hyp}{Hypothesis}{Hypotheses}
\Crefname{hyp}{Hypothesis}{Hypotheses}
\newtheorem{asu}{Assumption}
\newcounter{subassumption}[asu]
\renewcommand{\thesubassumption}{\quad \small $\bullet$}
\makeatletter
\renewcommand{\p@subassumption}{\theasu}
\makeatother
\newcommand{\subasu}{
	\refstepcounter{subassumption}%
	\thesubassumption~\ignorespaces}
\crefname{lemma}{Lemma}{Lemmas}
\Crefname{lemma}{Lemma}{Lemmas}
\crefname{theorem}{Theorem}{Theorems}
\Crefname{theorem}{Theorem}{Theorems}
\crefname{asu}{Assumption}{Assumptions}
\Crefname{asu}{Assumption}{Assumptions}
\crefname{subassumption}{Assumption}{Assumptions}
\Crefname{subassumption}{Assumption}{Assumptions}
\crefname{subasu}{Assumption}{Assumptions}
\Crefname{subasu}{Assumption}{Assumptions}
\crefname{section}{Section}{Sections}
\Crefname{section}{Section}{Sections}
\crefname{proposition}{Proposition}{Propositions}
\Crefname{proposition}{Proposition}{Propositions}

\newcommand{\absd}{\vert d\vert}
\newcommand{\Mostafa}[1]{\textcolor{black}{#1}}
\theoremstyle{thmstyleone}%
\newtheorem{theorem}{Theorem}
\newtheorem{proposition}[theorem]{Proposition}%
\newtheorem{lemma}{Lemma}%

\theoremstyle{thmstyletwo}%
\newtheorem{remark}{Remark}%

\theoremstyle{thmstylethree}%

\begin{document}
	\title[Analysis of shape optimization problems in nonlinear acoustics]{\vspace*{-1.2cm}Analysis of general shape optimization problems in nonlinear acoustics}
	
	\author[M. Meliani and V. Nikoli\'{c}]{\small Mostafa Meliani$^*$ and Vanja Nikoli\'{c}}
	\address{ 
		Department of Mathematics,
		Radboud University   \\ 
		Heyendaalseweg 135,
		6525 AJ Nijmegen, The Netherlands}
	\thanks{$^*$Corresponding author: M.Meliani, \href{mailto:mostafa.meliani@ru.nl}{mostafa.meliani@ru.nl}}
	\vspace*{8mm}
	\vspace*{-7mm}
	\begin{abstract}
	In various biomedical applications, precise focusing of nonlinear ultrasonic waves is crucial for efficiency and safety of the involved procedures. This work analyzes a class of shape optimization problems constrained by general quasi-linear acoustic wave equations that arise in high-intensity focused ultrasound (HIFU) applications. 
	Within our theoretical framework, the Westervelt and Kuznetsov equations of nonlinear acoustics are obtained as particular cases. The quadratic gradient nonlinearity, specific to the Kuznetsov equation, requires special attention throughout. To prove the existence of the Eulerian shape derivative, we successively study the local well-posedness and regularity of the forward problem, uniformly with respect to shape variations, and prove that it does not degenerate under the hypothesis of small initial and boundary data. Additionally, we prove H\"older-continuity of the acoustic potential with respect to domain deformations. We then derive and analyze the corresponding adjoint problems for several different cost functionals of practical interest and conclude with the expressions of well-defined shape derivatives. 
	\end{abstract}
	\maketitle
	\vspace*{-4mm}  
\maketitle
\vspace*{-0.4cm}

\section{Introduction}
In 1848, Stokes published an article on a difficulty in the theory of sound, namely the difficulty to apprehend the nonlinear behavior of acoustic waves \cite{stokes1848liv}. 
Nonlinear wave phenomena still pose challenging questions to this day. 
In various biomedical applications~\cite{muir1980nonlinear,muir1980prediction,kennedy2003high,wu2001pathological,yoshizawa2009high}, 
understanding and manipulating nonlinear ultrasonic waves is essential in ensuring the efficiency and safety of the involved procedures. 
Among such procedures, High-Intensity Focused Ultrasound (HIFU) is emerging as one of the most promising non-invasive tools in treatments of various solid
 cancers~\cite{zhou2011high,maloney2015emerging,kennedy2003high}. However, its wide-scale use hinges on the ability to guarantee the desired sound 
  behavior in the focal region. \Mostafa{Other applications of HIFU include, among others, lithotripsy~\cite{yoshizawa2009high} and stroke treatments~\cite{pajek2012design}.}\\
 \indent \Mostafa{HIFU} waves are commonly excited by one or several piezoelectric transducers arranged on a spherical 
  surface~\cite{martins2012optimization,pajek2012design,rosnitskiy2016design}. Changes in their shape directly affect the propagation and focusing of sound waves. 
  \Mostafa{Depending on the application, the requirements for the pressure levels and wave distortion at the focal region differ. 
  Designing the transducer arrangement for a given set of requirement gives rise to a practically relevant optimization problems.} \\
\indent Motivated by this, in the present work we conduct the analysis of a class of shape optimization problems subject to the following model of ultrasound propagation:
\begin{align} \label{eq:Kuzn_potential}
 (1-2k\dot \psi)\ddot \psi - c^2 \laplace \psi - b \laplace \dot \psi - 2\sigma\nabla\dot\psi \cdot \nabla\psi = 0,
\end{align}
\Mostafa{on a bounded smooth domain $\Omega$, }assuming nonhomogeneous Neumann boundary excitation. The particular choice of the parameters $k$ and $\sigma$ allows us to cover the strongly damped linear wave equation as well as the classical Westervelt~\cite{westervelt1963parametric} and Kuznetsov~\cite{kuznetsov1970equations} equations of nonlinear acoustics (in potential form)
 within our analysis. We refer to \cref{sec:wellp} below for a more detailed discussion of the studied model.\\
\indent  The well-posedness and regularity of nonlinear acoustic waves on fixed domains have been by now widely studied in the literature; see, e.g.,~\cite{kaltenbacher2009global, kaltenbacher2011well, kaltenbacher2019vanishing, dekkers2017cauchy, KALTENBACHER20191595, kaltenbacher2016shape, kiyoshi1993, ang1988strongly, meyer2011optimal}. 
 Shape differentiability of acoustic wave equations and related optimization problems have been mostly examined subject to linear evolution; see, e.g.,~\cite{cagnol1999shape, bociu2017hyperbolic, hinz2021non} and the references given therein. The few results~\cite{kaltenbacher2011sensitivity, kaltenbacher2016shape, muhr2017isogeometric, nikolic2017sensitivity} available in shape optimization with quasilinear acoustic models focus on the less involved Westervelt equation in pressure form. \\
\indent In this paper, we enrich the existing literature with the study of a quasilinear model \eqref{eq:Kuzn_potential} with general quadratic nonlinearities (i.e., the terms $\dot{\psi} \ddot{\psi}$ and $\nabla\dot\psi \cdot \nabla\psi$). 
 The presence of these terms necessitates the use of higher-order energies in the analysis of the state problem compared to~\cite{kaltenbacher2016shape, muhr2017isogeometric}. In turn, new arguments need to be devised in the study of H\"older-continuity with respect to shape variations. 
 Furthermore, the quadratic gradient nonlinearity leads to adjoint problems 
 \Mostafa{ with boundary conditions in the form:}
 \Mostafa{ 
 \[
 c^2\pdfrac{p}{n}-b\pdfrac{\dot p}{n} +2 \sigma \pdfrac{\psi}{n} \dot p = 0,
 \]
where $p$ is the adjoint variable. When $\sigma \neq 0$, such a boundary condition requires special attention in the well-posedness analysis.
 }

Besides, we not only treat the classical 
 $L^2(L^2)$-tracking problem on  $D = D_S\times (t_0,t_1)$, where the acoustic velocity potential $\psi$ should match a desired output on a given spatial focal region $D_S \Mostafa{\subsetneq \Omega} $ within a certain time interval $(t_0, t_1)$, we also consider in this work an $L^2$-matching objective at final time:
 \begin{align} \label{objective_2}
	J_T(\psi,\Omega) = \frac{1}{2}\int_{\Omega} \left(\psi(T) - \psi_{D_S}\right)^2 \chi_{D_S} \dx.
	\end{align}\\[2mm]
\noindent Additionally, we study an $L^2(L^2)$-tracking functional on $\psid$,  
corresponding (up to a multiplicative constant) to tracking the sound pressure\Mostafa{. We denote by $f_D$ the desired pressure on $D$. The cost functional is then given by}	
 \begin{align}\label{objective_3}
	J_\textup{p}(\psi,\Omega) = \frac{1}{2}\int_0^T\int_{\Omega} \left(\psid - f_D\right)^2 \chi_{D} \dxs;
\end{align}  
see Figure~\ref{fig:setup} for an illustration of the optimization setup.

The shape analysis is carried out by adopting the adjoint-based variational framework developed in \cite{ito2008variational}, 
which avoids the need to study the shape sensitivity of the state variable. Instead we rely on the aforementioned H\"older continuity of the potential with respect to domain perturbations in a suitable norm.
 We show that with appropriate (different) adjoint problems, these objectives have the same final derivative expression which is of the form required by the Delfour--Hadamard--Zol\'esio structure theorem; see~\cite[Theorem 3.6]{delfour2011shapes}. The well-posedness analysis of the state and adjoint problems ensures the well-definedness of the expression.
\begin{figure}[H]
\centering  
 \scalebox{0.85}
 {
 \begin{tikzpicture}[
        ->,    > = stealth',
       shorten > = 1pt,auto,
   node distance = 3cm,
      decoration = {snake,   
                    pre length=3pt,post length=7pt,
                    },
main node/.style = {circle,draw,fill=#1,
                    font=\sffamily\Large\bfseries},
                    ]
\draw[black, fill = white, dashed ,fill opacity = 0.5, semithick, even odd rule]
(-1,-1) rectangle (6.5,4);
\draw[black, fill = white, fill opacity = 0.5, semithick, even odd rule,rounded corners=10]
(0,0) rectangle (5.5,3);
\draw[black, fill = white, dashed,  fill opacity = 0.5, semithick, even odd rule] 
(2.75, 1.5) circle (0.8);
\node at (5.2, 2.5) {\large $\Omega$};
\node at (2.75, 1.5) {\large $D_S$};
\node at (-0.5,1.5) {\large $\partial\Omega$};
\node at (1,-0.5) {\large hold-all domain $U$};
\node[align=center] at (1,2.3) {Propagating \\ wave};
\node (1) at (5.6, 1.5) {};
\node (2) at (4, 1.5) {};
\node (3) at (-0.1, 1.7) {};
\node (4) at (1.5, 1.7) {};
\path[ draw = blue , decorate] (1)  -- node[auto] {} (2);
\path[ draw = blue , decorate] (3)  -- node[auto] {} (4);
\end{tikzpicture}
 } 
~\\[2mm]
\caption{The optimization setup, where $D = D_S \times (t_0,t_1)$}
\label{fig:setup}
\end{figure}
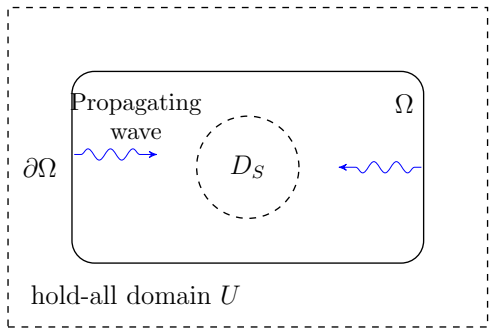
\vspace*{-2mm}

The paper is organized through multiple sections, each dealing with questions relating to the existence and well-definedness of the derivative.
In \cref{sec:wellp}, we discuss the well-posedness of the forward problem under Neumann boundary conditions and the regularity of its solution, also uniformly with respect to shape deformations. In \cref{sec:Holder_cont}, we prove the H\"older continuity of the velocity potential with respect to domain perturbations.
We then study the well-posedness and regularity of the adjoint problem for the $L^2(L^2)$-tracking objective. In \cref{sec:shapeanalysis}, we rigorously derive the shape derivative in the direction of any sufficiently smooth vector field deformation. Additionally, we give a brief analysis of the $L^2$-matching objective at final time \eqref{objective_2} and of the pressure cost functional \eqref{objective_3} and give their shape derivative expression.
 \subsection*{Notation} 
We shall frequently use the notation $x \lesssim y$ which stands for $x \leq Cy$, where $C$ is a generic constant that depends only on the reference domain $\Omega$ and, possibly, the final time $T$. We will denote by $\co$ the embedding constant of $H^1(\Omega)$ into $L^4(\Omega)$  and by $\ctr$ the norm of the trace operator
 $\operatorname{tr_{\bndry}}: H^1(\Omega) \rightarrow H^{1/2}(\bndry)$.
We denote by $\dot x$, $\ddot x$ and $\dddot x$ the first, second and third, respectively, derivative of the quantity $x$ with respect to time. Given two Banach spaces $X$ and $Y$, we shall refer to the continuous (resp. compact) embedding of $X$ into $Y$ by $X\hookrightarrow Y$ (resp. $X\dhookrightarrow Y$).

\section{Analysis of the state problem} \label{sec:wellp}
In this section, we discuss the well-posedness of \eqref{eq:Kuzn_potential}
as well as of its more general counterpart with variable smooth coefficients; see~\eqref{eq:holder2} below. \\
\indent Throughout this work we assume that $\Omega$ is a $C^{2,1}$-regular, bounded domain in $\Rs^l$ with $l\in \{2,3\}$. 
In \eqref{eq:Kuzn_potential}, $\psi$ denotes the acoustic velocity potential. It is related to the acoustic pressure $\textup{p}$ via 
\[\textup{p}= \rho \dot{\psi},\]
where $\rho$ is the mass density of the medium. The constant $c>0$ denotes the speed of sound and $b>0$ is the so-called sound diffusivity. Note that the presence of the strong damping ($b>0$) in the  equation contributes to its parabolic-like character; see, e.g,~\cite{kaltenbacher2011well, kiyoshi1993}. The
constant $k$ is a function of the coefficient of nonlinearity of the medium $\beta_a$ and of the speed of propagation $c$.

Equation \eqref{eq:Kuzn_potential} can be seen as the Westervelt or Kuznetsov equation of nonlinear acoustics, 
depending on the choice of parameters $k$ and $\sigma$. Indeed,
\[
(k,\sigma)=\left\{
\begin{array}{ll} 
\left(\frac{\beta_a}{c^2},0\right) 
&\textrm{for the Westervelt equation~\cite{westervelt1963parametric}},\\
\\
\left(\frac{\beta_a-1}{c^2},1\right)
& \textrm{for the Kuznetsov equation~\cite{kuznetsov1970equations}}.
\end{array}
\right.
\]
The choice between these two models depends on whether one can assume that cumulative nonlinear effects dominate local nonlinear effects; see the discussion in~\cite[Ch.\ 3, Section 6]{hamilton1998nonlinear}.
 We refer to \cite{jordan2016survey} for a survey of models in nonlinear acoustics. Going forward, we assume that $k$, $\sigma \in \Rs$. \\
\indent We couple \eqref{eq:Kuzn_potential} with an excitation 
modeled by Neumann boundary conditions as well as the initial data and study the following problem:

\begin{empheq}{align}
\label{pb:wellposedness}
	\begin{array}{rl}
	\left(1-2k\dot \psi\right)\ddot \psi - c^2 \laplace \psi - b \laplace \dot \psi 
	- 2\sigma\nabla\dot\psi\cdot\nabla\psi= 0 &\quad \mathrm{on\;} Q, \\
	\pdfrac{\psi}{n} =  g &\quad \mathrm{on\;} \Sigma, \\
	\psi(0) = \psi_0,\; \dot\psi(0) = \psi_1  &\quad \mathrm{on\; \Omega},\\
	\end{array}
\end{empheq}
where we have denoted $Q =  \Omega \times (0,T)$ and $\Sigma= \partial \Omega \times (0,T)$. Mathematical analysis of the Kuznetsov equation with homogeneous Dirichlet data can be found in~\cite{kiyoshi1993}. 
The Kuznetsov equation in pressure form with nonhomogeneous Neumann data is studied in~\cite{kaltenbacher2011well}. 
The analysis of \eqref{pb:wellposedness} follows along similar lines with main differences in the energy arguments, which we present here. For the sake of readability, we first study the well-posedness of \eqref{pb:wellposedness} and then generalize the argument to the case of variable coefficients resulting from domain perturbations; see \cref{prop:uniformwellp} below.

\begin{asu}\label{asu:ori} We make the following assumptions on the regularity and compatibility of data in \eqref{pb:wellposedness}:\\[1mm]
  \subasu \label{subasu:initori} $\psi_0 \in H^3(\Omega)$ and $\psi_1 \in H^2(\Omega)$;\\[1mm]
  \subasu \label{subasu:gspaceori}$g \in H^2(0,T;H^{-1/2}(\bndry)) \cap H^1(0,T;H^{3/2}(\bndry)) $, \\ \hphantom{(H2)}\ $\dot g \in L^\infty(0,T; H^{1/2}(\bndry))$; \\[1mm]
  \subasu \label{subasu:compatibilty} $\dfrac{\partial \psi_0}{\partial n} = g(0)$, $\dfrac{\partial \psi_1}{\partial n} = \dot g(0).$ 
\end{asu}

\noindent \textbf{Auxiliary inequality.} We recall that the following embeddings hold: $H^1(\Omega) \dhookrightarrow L^4 (\Omega) \hookrightarrow L^2 (\Omega)$; see~\cite[Theorem 1.3.2]{zheng2004nonlinear}. Let $u\in H^1(\Omega)$. 
Then by Ehrling's Lemma \cite[Theorem 7.30]{renardy2006introduction} and Young's inequality, for all $\epsilon >0$, 
there exists a constant $\ce(\epsilon)$, such that
\begin{align}\label{eq:usefulineq}
\normso{u}{4}^2 \leq \ce(\epsilon)\normso{u}{2}^2 + \epsilon \normso{\nabla u}{2}^2.
\end{align}
This result will be used throughout the analysis to avoid relying on $\normso{u}{4}^2 \leq \co^2 \normho{u}{1}^2$ when critical. 
This small change allows us to avoid imposing additional assumptions on the smallness of 
$\ddot\psi$ below (as done, for example, on $\dot y$ in \cite[Theorem 1]{kaltenbacher2016shape}, where the Westervelt equation in pressure form is analyzed).\\[2mm]
\noindent \textbf{Linearized equation.} Following the general approach of, e.g., \cite{kaltenbacher2016shape, kaltenbacher2011well}, we first consider a linearized problem 
with a variable coefficient $a$ and a source term $f$:
\begin{empheq}{align}\label{eq:linearized1}
	\begin{array}{rl}
		a(t)\ddot \xi - c^2 \laplace \xi 
			 - b\laplace \dot \xi  = f & \qquad \mathrm{on\;}\ Q,\\
	\Mostafa{\pdfrac{\xi}{n} = g  }&\Mostafa{\qquad \mathrm{on\;} \Sigma,}\\
	\Mostafa{\xi(0) = \psi_0,\; \dot\xi(0) = \psi_1}&\Mostafa{\qquad \mathrm{on\; \Omega}},
	\end{array}
\end{empheq}
supplemented by the same initial and boundary conditions as in \eqref{pb:wellposedness}. To study such a problem, we need suitable assumptions on the non-degeneracy and smallness of $a$.
\begin{asu}\label{asu:asulin} We make the following additional assumptions on the source term and the variable coefficient in \eqref{eq:linearized1}:\\[2mm]
  \subasu \label{subasu:boundedness}$a \in L^\infty(0,T;H^2(\Omega))\cap L^2(0,T;H^3(\Omega))$, $\dot a \in L^\infty(0,T;L^2(\Omega))$.
   Furthermore, \Mostafa{ let there exist $\ula$, $\ola > 0$ such that} 
 \[\Mostafa{\ula}<a(t)< \Mostafa{\ola} \ \text{almost everywhere on } \ \Omega \ \text{ for almost every } \ t\in[0,T];\] 
  \subasu \label{subasu:fspace}$f \in L^2(0,T;H^1(\Omega)) \cap W^{1,1}(0,T;L^2(\Omega))$. 
\end{asu}

\begin{proposition}\label{th:linearized}
Let \cref{asu:ori,asu:asulin} hold. Then initial boundary-value problem for \eqref{eq:linearized1} admits a unique solution in
\begin{align} \label{def_W}
\begin{split}
W = \{\xi\in L^\infty(0,T, H^3(\Omega)):\;\dot \xi \in L^\infty(0,T, H^2(\Omega)) \cap L^2(0,T, H^3(\Omega)) ,\; 
\\ \ddot \xi\in L^\infty(0,T, L^2(\Omega)) \cap L^2(0,T, H^1(\Omega))\}.
\end{split}
\end{align}
Furthermore, the solution satisfies the energy estimate 
\begin{align}\label{eq:energynormth}
\begin{split}
\left\| \xi\right\|_W^2
\lesssim \; & \begin{multlined}[t] \normho{\psi_0}{3}^2 + \normho{\psi_1}{2}^2 + \normtho{f}{2}{1}^2 
+ \normtso{\dot f}{1}{2}^2 \\+\normthg{g}{1}{3/2}^2 + \normthg{g}{2}{-1/2}^2+ \normtlhg{\dot g}{\infty}{1/2}^2 . 
\end{multlined}
\end{split}
\end{align}
\end{proposition}
\begin{remark}
The regularity of $\xi \in W$ implies that $\xi$ is continuous from $[0,T]$ into $H^3(\Omega)$ and $\dot\xi$ is weakly continuous from $[0,T]$ into $H^2(\Omega)$ \cite[Lemma 3.3]{temam2012infinite}. Note that the hidden constant in \eqref{eq:energynormth} depends on $a$ and tends to $+ \infty$ as $T \rightarrow + \infty$.
\end{remark}
\begin{proof}
The proof can be conducted rigorously through a Galerkin approximation of the weak form of the PDE by restriction to appropriate finite dimensional subspaces of $H^3(\Omega)$; see, for example, 
\cite[Chapter 7]{evans10} and \cite[Chapter XVIII]{dautray1999mathematical}. We focus here on the derivation of a uniform energy estimate, which differs from the related results in the literature (cf.~\cite[Theorem 1]{kaltenbacher2016shape}) and will enable the upcoming fixed-point argument. For notational simplicity, we omit the index indicating the Galerkin approximation below. \\

\noindent\textbf{Lower-order energy estimate.}
To obtain the desired energy estimate, we test the semi-discrete problem by $\dot\xi$ and integrate in time over $(0,t)$ to obtain
\begin{multline*}
\ints \left[ (a(s)\ddot \xi ,\dot \xi) +c^2 (\nabla \xi, \nabla \dot\xi) + b (\nabla \dot \xi, \nabla \dot\xi) \right]\dt \\
= \ints (f(s),\dot\xi)\dt + \ints\intg (c^2g + b \dot g) \dot\xi\ds\dt,
\end{multline*}
which, after using the relation 
\begin{align} \label{eq:energy_id_1}
a\ddot \xi \dot \xi = \frac{1}{2} \diffrac{}{t}(a \dot \xi^2) - \frac{1}{2} (\dot a \dot \xi^2),
\end{align}
yields the energy identity
\begin{align*}
\begin{split}
&\frac{1}{2} \|\sqrt{a(t)}\dot\xi(t)\|^2_{L^2(\Omega)} + \frac{c^2}{2} \|\nabla \xi(t)\|^2_{L^2(\Omega)} 
+ b \ints \|\nabla \dot \xi\|^2_{L^2(\Omega)} \dt  \\= \;
& \frac{1}{2} \ints (\dot a \dot \xi, \dot \xi) \dt+ \frac{1}{2} \|\sqrt{a(0)}\xi_1\|^2_{L^2(\Omega)} + \frac{c^2}{2} \|\nabla \xi_0\|^2_{L^2(\Omega)}
+ \ints  (f(s),\dot\xi)  \dt \\
&+ \ints\intg (c^2g + b \dot g) \dot\xi\ds\dt.
\end{split}
\end{align*}
We can estimate the first term on the right using Ehrling's inequality \eqref{eq:usefulineq} as follows:
\begin{align*}
\begin{split}
\vline(\dot a \dot \xi, \dot \xi)\vline \leq \;& \normso{\dot a}{2} \normso{\dot \xi}{4}^2 \leq   \normso{\dot a}{2} 
\left( \ce(\epsilon)\normso{\dot\xi}{2}^2 + \epsilon \normso{\nabla \dot\xi}{2}^2\right).
\end{split}
\end{align*}
We then proceed to estimating the source term 
\begin{align*}
\ints (f,\dot\xi) \dt \leq \ints & \left[\frac{1}{2} \normso{f}{2}^2 +
\frac{1}{2} \normso{\dot\xi}{2}^2 \right]\dt,
\end{align*}
and the boundary term
\begin{align*}
\ints\intg (c^2g + b \dot g) \dot\xi\ds\dt \leq\; & \normtlhg{c^2g + b \dot g}{2}{-1/2}\left(\ints\normhg{\dot\xi}{1/2}^2 \dt\right)^{1/2}\\
\leq \; & \frac{1}{4\epsilon} \normtlhg{c^2g + b \dot g}{2}{-1/2}^2 + \epsilon \ctr^2 \ints\normho{\dot\xi}{1}^2\dt.
\end{align*}
Using \cref{asu:asulin}, we then have
\begin{align} \label{eq:lowenergy1}
\begin{split}
& \Mostafa{\frac{\ula}{2}}  \normso{\dot\xi(t)}{2}^2 + \frac{c^2}{2} \normso{\nabla \xi(t)}{2}^2 
\\ &+ \left(b -\epsilon\frac{\normtso{\dot a}{\infty}{2}+2 \ctr^2}{2} \right) \ints \normso{\nabla \dot \xi}{2}^2 \dt 
\\\leq\;
& \frac{\ce(\epsilon)\normtso{\dot a}{\infty}{2}+1 +2\epsilon \ctr^2}{2} \ints \normso{\dot \xi}{2}^2 \dt  +
\Mostafa{\frac{\ola}{2}} \normso{\dot\xi(0)}{2}^2 \\
&+ \frac{c^2}{2} \normso{\nabla\xi(0)}{2}^2
+\frac{1}{2}  \ints  \normso{f}{2}^2  \dt + \frac{1}{4\epsilon} \normtlhg{c^2g + b \dot g}{2}{-1/2}^2,
\end{split}
\end{align}
where $\epsilon>0$ is chosen small enough so that  
\begin{align} \label{smallness_eps}
b -\epsilon\frac{\normtso{\dot a}{\infty}{2}+2 \ctr^2}{2}  >0 .
\end{align}

To get the next component of the desired energy estimate, we differentiate the semi-discrete problem with respect to time and test with $\ddot \xi$, which, 
after integration by parts in space, yields 
\begin{multline*}
\ints \left[ ( a(s)\dddot \xi ,\ddot \xi) +c^2 (\nabla \dot\xi, \nabla\ddot\xi) + b (\nabla \ddot \xi,  \nabla\ddot\xi) \right]\dt 
\\= - \ints (\dot a(s)\ddot \xi ,\ddot \xi) \dt+\ints (\dot f(s),\ddot\xi)\dt + \ints\intg (c^2 \dot{g}+b \ddot{g}) \ddot\xi\ds\dt.
\end{multline*}
We estimate the source term, for any $\epsilon_0>0$, as
\begin{align*}
\ints (\dot f(s),\ddot\xi)\dt \leq  \frac{1}{4\epsilon_0} \left(\ints \normso{\dot f(s)}{2} \dt\right)^2 + \epsilon_0 \sup_{\tau\in (0,t)} \normso{\ddot \xi(\tau)}{2}^2,
\end{align*}
We can then infer the following inequality: 
\begin{align*} 
\begin{split}
&\Mostafa{\frac{\ula}{2}}\normso{\ddot \xi(t)}{2}^2 + \frac{c^2}{2} \normso{\nabla \dot\xi(t)}{2}^2 + b \ints \normso{\nabla \ddot \xi}{2}^2 \dt 
\\ \leq \;&\Mostafa{\frac{\ola}{2}}\normso{\ddot \xi(0)}{2}^2  + \frac{c^2}{2} \normso{\nabla \dot\xi(0)}{2}^2 - 
\frac{1}{2}\ints (\dot a(s)\ddot \xi ,\ddot \xi) \dt \\&+ \frac{1}{4\epsilon_0} \left(\ints \normso{\dot f(s)}{2} \dt\right)^2 
+ \epsilon_0 \sup_{\tau\in (0,t)} \normso{\ddot \xi(\tau)}{2}^2
+ \ints\intg (c^2 \dot{g}+b \ddot{g}) \ddot\xi\ds\dt,
\end{split}
\end{align*}
where $\ddot \xi(0)$ can be estimated by testing \eqref{eq:linearized1} by $\ddot\xi$ at $t=0$:
\begin{align}\label{eq:xidd_estimate}
\frac{1}{2} \normso{\ddot\xi(0)}{2} \leq  c^2\normso{\laplace \xi(0)}{2} + b \normso{\laplace \dot \xi(0)}{2} + \normso{f(0)}{2}.
\end{align}
Notice that due to the embedding $W^{1,1}(0,T)\hookrightarrow C[0,T]$ (see \cite[Ch.\ 8, p.\ 214]{brezis2010functional}), the term $\normso{f(0)}{2}$ can be estimated with $\normfspe$.

\noindent We develop the rest of the estimates analogously to before and obtain
\begin{align} \label{eq:4th2nd}
\begin{split}
&\Mostafa{\frac{\ula}{2}} \normso{\ddot \xi(t)}{2}^2 + \frac{c^2}{2} \normso{\nabla \dot\xi(t)}{2}^2 \\
&+ \left(b -\epsilon\frac{\normtso{\dot a}{\infty}{2}+2 \ctr^2}{2} \right) \ints \normso{\nabla \ddot \xi}{2}^2 \dt 
\\ \leq \; & \frac{\ce(\epsilon)\normtso{\dot a}{\infty}{2} +2\epsilon \ctr^2}{2}\ints \normso{\ddot\xi}{2}^2\dt+  \Mostafa{\frac{\ola}{2}}\normso{\ddot\xi(0)}{2}^2 
\\& + \frac{c^2}{2} \normso{\nabla \dot\xi(0)}{2}^2 
+ \frac{1}{4\epsilon_0} \left(\ints \normso{\dot f(s)}{2} \dt\right)^2  +\epsilon_0 \sup_{\tau\in (0,t)} \normso{\ddot \xi(\tau)}{2}^2 
\\& + \frac{1}{4\epsilon} \normtlhg{c^2 \dot g + b \ddot g}{2}{-1/2}^2, 
\end{split}
\end{align}
with $\epsilon>0$ chosen again so that \eqref{smallness_eps} holds.

We can then combine \eqref{eq:lowenergy1} and \eqref{eq:4th2nd}, taking the supremum over $t \in (0, \tilde{t})$ for some $0<\tilde{t}\leq T$ which allows us to absorb the $\epsilon_0$ term and then use Gronwall's inequality to eliminate the energy terms form the right-hand side. Together with estimate \eqref{eq:xidd_estimate} and the fact that
\begin{align*}
\normso{\xi (\tilde{t})}{2} = \normso{\xi(0) + \int_0^{\tilde{t}} \dot \xi (s) \dt}{2} \leq \normso{\xi_0}{2} + \tilde{t}
\normtso{\dot \xi}{\infty}{2},
\end{align*}
we obtain the following bound on the energy norm:
\begin{equation}\label{eq:energynorm1}
\begin{aligned}
&\normho{\xi (\tilde{t})}{1} ^2 +\normho{\dot\xi(\tilde{t})}{1}^2 + \normso{\ddot \xi(\tilde{t})}{2}^2 
+ \int_0^{\tilde{t}} \normso{\nabla \ddot \xi(s)}{2}^2 \dt
\\ \lesssim& \;\begin{multlined}[t]  \normho{\psi_0}{2}^2 + \normho{\psi_1}{2}^2 + \normtso{f}{1}{2}^2  +  \normtso{\dot f}{1}{2}^2\\+ \normthg{g}{2}{-1/2}^2,
\end{multlined}
\end{aligned}
\end{equation}
a.e.
 in time, where the hidden constant is given by 
\begin{align}\label{eq:constant_expr}
C_1(1+\tilde{t})\exp\left(C_2(1+\normtso{\dot a}{\infty}{2})\tilde{t}\right) ,
\end{align}
where $C_1$ and $C_2$ are two positive constants. \\[2mm]
\noindent\textbf{Bootstrap argument.}
Inspired by the approach of, e.g., \cite[Lemmas 3.1 and 3.2]{garcke2018optimal}, we demonstrate, through a bootstrap argument, that a higher regularity is achieved.
To do so, we rewrite \eqref{eq:linearized1} as
\begin{empheq}{align*}
	\begin{array}{rl}
		 -  \laplace \left(c^2\xi +b \dot \xi \right) + c^2\xi +b \dot \xi = f - a(t)\ddot \xi + c^2\xi +b \dot \xi  & \qquad \mathrm{on\;} Q,\\
	c^2 \pdfrac{\xi }{n} + b \pdfrac{\dot \xi}{n} = c^2 g +b\dot g &\qquad \mathrm{on\;} \Sigma,\\
	\xi(0) = \xi_0,\; \dot\xi(0) = \xi_1&\qquad \mathrm{on\; \Omega}.
	\end{array}
\end{empheq}
Let $\Theta = c^2 \xi(t) + b\dot \xi(t)$ and $\tilde{f} = f(t) - a(t)\ddot \xi(t) + c^2\xi(t) +b \dot \xi(t)$. Then $\Theta$ solves 
\begin{empheq}{align*}
	\begin{array}{rl}
		 -  \laplace \Theta + \Theta = \tilde{f}  & \qquad \mathrm{on\;} \Omega,\\
	 \pdfrac{\Theta }{n} = c^2g+b\dot g &\qquad \mathrm{on\;} \bndry.\\
	\end{array}
\end{empheq}
By estimate \eqref{eq:energynorm1} and \cref{asu:asulin}, we infer that $\tilde{f} \in L^2(0,T;H^1)$. Recalling that $\Omega$ is of class $C^{2,1}$, elliptic regularity yields 
\begin{align*}
\normho{\Theta}{3} \lesssim \normho{\tilde{f}}{1} + \normhg{c^2g+b\dot g}{3/2};
\end{align*}
see~\cite[Theorem 2.5.1.1]{grisvard2011elliptic}. Since we know that $\xi$ solves the ODE 
$$ \dot \xi + \frac{c^2}{b} \xi = \frac{\Theta}{b},$$
we have 
\begin{align*}
\xi = e^{-c^2t/b} \left(\xi_0 + \ints e^{c^2s/b} \frac{1}{b} \Theta(s) \dt\right)
\end{align*}
and subsequently 
\begin{align*}
\normtho{\xi}{\infty}{3} \lesssim & \normho{\xi_0}{3} + \normtho{\Theta}{2}{3} \\
\lesssim & \normho{\xi_0}{3} + \normtho{\tilde{f}}{2}{1}+ \normtlhg{c^2g+b\dot g}{2}{3/2}.  
\end{align*}
Furthermore, 
\begin{align*}
\normtho{\dot \xi}{2}{3} =&\,  \frac{1}{b}\normtho{\Theta - c^2\xi}{2}{3} \\
\lesssim&\,  \normho{\xi_0}{3}   + \normtho{\tilde{f}}{2}{1}+ \normtlhg{c^2g+b\dot g}{2}{3/2},
\shortintertext{and similarly,}
\normtho{\dot \xi}{\infty}{2} =&\,  \frac{1}{b}\normtho{\Theta - c^2\xi}{\infty}{2} \\
\lesssim&\,  \normho{\xi_0}{2}   + \normtso{\tilde{f}}{\infty}{2}+ \normtlhg{c^2g+b\dot g}{\infty}{1/2}.
\end{align*}
Together, these estimates yield $\xi \in W$ and
\begin{multline*}
\left\| \xi\right\|_W^2
\lesssim  \normho{\psi_0}{3}^2 + \normho{\psi_1}{2}^2 + \normtho{f}{2}{1}^2 
+  \normtso{ \dot f}{1}{2}^2 \\+\normthg{g}{1}{3/2}^2 + \normthg{g}{2}{-1/2}^2 +\normtlhg{\dot g}{\infty}{1/2} .
\end{multline*}
With this uniform bound in hand, standard compactness arguments allow us to extract 
weakly convergent subsequences that verify the PDE and its initial conditions; see e.g.,~\cite{evans10,salsa2016partial}. We omit the details here.
\end{proof}
\noindent We next move on to the analysis of the Kuznetsov equation where we show that the aforementioned regularity and well-posedness results hold for the nonlinear equation, provided smallness of data, using a fixed-point argument. To this end, we define the operator 
\begin{align}\label{eq:fixptopdef}
	\mathrm{A}[\cdot]: \mathcal{W} \ni y \mapsto &\:\psi,
\end{align}
where $\psi$ is the solution of the initial boundary value problem 
\begin{empheq}{align*}
\begin{array}{rl}
(1-2k \dot y)\ddot \psi - c^2 \laplace \psi 
- b\laplace \dot \psi  = 2 \sigma \nabla \dot y \cdot \nabla y  & \qquad \mathrm{on\;} Q,\\
\pdfrac{\psi}{n} = g  &\qquad \mathrm{on\;} \Sigma,\\
\psi(0) = \psi_0,\; \dot\psi(0) = \psi_1&\qquad \mathrm{on\; \Omega},
\end{array}
\end{empheq}
and the ball $\mathcal{W}$ is defined by 
\begin{equation}\label{eq:Wballdef}
\begin{aligned}
\begin{multlined}[t]
\mathcal W  = \left\{\vphantom{\frac{1}{4\vline k\vline}} y \in W \right. :\,  
\normtso{\dot y}{\infty}{\infty} \leq \frac{1}{4\vert k \vert},\ \normtho{\nabla y}{\infty}{2}\leq r,\ \\ \normtho{\nabla \dot y}{\infty}{1}\leq r, \\
\;\normtho{\ddot y}{2}{1}\leq r,\;  \normtso{\ddot y}{\infty}{2}\leq r, \left.\; \normtso{\nabla \dot y}{2}{\infty} \leq r \vphantom{\frac{1}{4 \vline k \vline}}  \right\},
\end{multlined}
\end{aligned}
\end{equation}
for some $r>0$; recall that the space $W$ is defined in \eqref{def_W}. The choice of $r>0$ will be made below. Note that the set $\mathcal W$ is nonempty as $0 \in \mathcal W$.
\begin{theorem}\label{th:nonlinreg}
Let $T>0$. Under \cref{asu:ori}, there exists $r_0^2=r_0^2(\Omega, T)$, such that if
\begin{align*} 
\normho{\psi_0}{3}^2 + \normho{\psi_1}{2}^2 +\normthg{g}{1}{3/2}^2 + \normthg{g}{2}{-1/2}^2& \\+ \normtlhg{\dot g}{\infty}{1/2}^2 \leq r_0^2&,
\end{align*}
then the initial-boundary value problem  \eqref{pb:wellposedness} for the Kuznetsov equation has a unique solution $\psi \in W$. Furthermore, this solution satisfies 
\begin{align*}
\left\| \psi\right\|_W^2
\lesssim \; & \begin{multlined}[t]\normho{\psi_0}{3}^2 + \normho{\psi_1}{2}^2 +\normthg{g}{1}{3/2}^2 + \normthg{g}{2}{-1/2}^2 \\+ \normtlhg{\dot g}{\infty}{1/2}^2.\end{multlined}
\end{align*}
\end{theorem}

\begin{proof}
The proof follows by employing the Banach fixed-point theorem on the mapping $\mathrm{A}[\cdot]$.  We can use the existence and uniqueness result of \cref{th:linearized} to conclude that $\mathrm{A}[\cdot]$ is well defined by setting \[y \in W,\qquad f =  2\sigma \nabla \dot y \cdot \nabla y, \qquad a(t) = 1-2k\dot y.\] 
We see that $\nabla a \in L^2(0,T;H^1(\Omega))$ and $\dot a \in L^2(0,T;H^1(\Omega)) \cap L^\infty(0,T;L^2(\Omega))$. 
\Mostafa{Moreover, with the choice of the ball $\mathcal W$, $ \ula \leq a(t) \leq \ola$ is verified with $\ula = \frac12$ and $\ola = \frac32$.}
We next verify the regularity assumption on $f$ in \cref{th:linearized}. Denote by $\mathrm{H}(\cdot)$ the Hessian matrix operator of a scalar field. Then 
\begin{align}\label{eq:est_L2_f}
\begin{split}
&\ints \normho{f}{1}^2 \dt \\
 =\;& 4\sigma ^2 \ints \normso{\nabla \dot y \cdot\nabla y }{2}^2 \dt + 4\sigma ^2 \ints \normso{\nabla(\nabla \dot y\cdot\nabla y) }{2}^2 \dt
\\\lesssim\; &\begin{multlined}[t] \normtho{\nabla \dot y}{\infty}{1}^2 \normtho{\nabla y}{2}{1}^2 + \normtso{\nabla y}{2}{\infty}^2 \normtso{\mathrm{H}(\dot y)}{\infty}{2}^2 \\
+
\normtho{\nabla \dot y}{\infty}{1}^2 \normtho{\mathrm{H}(y)}{2}{1}^2. \end{multlined}
\end{split}
\end{align}
Similarly, we have 
\begin{align}\label{eq:est_L1_dotf}
\begin{split}
 \normtso{ \dot f}{1}{2}^2
\lesssim\normtho{\nabla \dot y}{2}{1}^4+\normtso{\nabla \ddot y}{2}{2}^2\normtso{\nabla y}{2}{\infty}^2. 
\end{split}
\end{align}
Therefore, there exists $C=C(t)>0$, such that
\begin{align} \label{est_f}
\ints \normho{f(s)}{1}^2 \dt + \left(\ints \normso{ \dot f(s)}{2}\dt \right)^2 \leq Cr^4.
\end{align}
Thus, \eqref{eq:fixptopdef} is well defined, provided $\normtso{\dot y}{\infty}{\infty} \leq \frac{1}{4\vert k \vert}$.
Next we demonstrate that $\mathcal{W}$ is invariant under a smallness condition on the problem data. By the energy bound \eqref{eq:energynormth}, there exists $C_3 =C_3(r)>0$, such that
\begin{multline*}
\|A[y]\|_W^2 \leq C_3  \left(\normho{\psi_0}{3}^2 + \normho{\psi_1}{2}^2 +  \normtho{f(s)}{2}{1}^2 \right. 
+ \normtso{ \dot f(s)}{1}{2}^2\\+\normthg{g}{1}{3/2}^2 + \normthg{g}{2}{-1/2}^2 +\left. \vphantom{\intt}\normtlhg{\dot g}{\infty}{1/2}^2\right).
\end{multline*}
We can estimate the norm terms involving $f$ as in \eqref{est_f}. Together, this yields that 
$$\|A[y]\|_W^2 \lesssim r_0^2 +r^4 .
$$
On account of the embedding $H^2(\Omega) \hookrightarrow L^\infty(\Omega)$,  the semi-norm $\vert\cdot\vert_{\mathcal{W}}$ defining the ball $\mathcal{W}$, is weaker than the norm $\|\cdot\|_W$. Therefore, there exists $C_4(r)>0$ depending continuously on $r$ (cf. \eqref{eq:constant_expr}), such that 
$$\vert A[y]\vert_{\mathcal{W}}^2 \leq C_4(r) (r_0^2 +r^4).$$
Since $C_4(r)>0$ depends continuously on $r$, for $r$ small enough, we have 
$$0<\frac12 C_4(0)<C_4(r)<\frac32 C_4(0).$$
It can then be readily shown that the set $\mathcal{W}$ is invariant under $A[\cdot]$ for $r_0$ sufficiently small if we impose additionally
$$r^2 < \frac{2}{3C_4(0)}.$$ 
Indeed, for $r_0 \leq \sqrt{\frac{2r^2}{3C_4(0)}-r^4}$, we have $\frac32C_4(0) \left(r_0^2 +r^4 \right)< r^2$. 

Next we prove the strict contractivity of the operator $\mathrm{A}[\cdot]$ in $$\hat{X} =H^2(0,T;L^2(\Omega)) \cap W^{1,\infty}(0,T;H^1(\Omega));$$
that is, in a weaker topology compared to the one of the solution space. Let $y$, $\tilde y \in W$ and let $\psi = \mathrm{A}[y]$ and $\tpsi = \mathrm{A}[\tilde{y}]$. We define $z = \psi-\tpsi$, which verifies 
\begin{empheq}{align}\label{eq:contractivity}
	\begin{array}{rl}
		 a(t) \ddot z- c^2 \laplace z  - b \laplace\dot z =  f - \tilde{f} - (a -\tilde{a}) \ddot \tpsi   & \qquad \mathrm{on\;} Q,\\[1mm]
	\pdfrac{z}{n} = 0 &\qquad \mathrm{on\;} \Sigma,\\[1mm]
	z(0) = 0,\; \dot z(0) = 0&\qquad \mathrm{on\; \Omega}.
	\end{array}
\end{empheq}
We test \eqref{eq:contractivity} with $\dot z$ and integrate in time over $(0,t)$ to obtain, analogously to estimate \eqref{eq:lowenergy1},
\begin{multline}\label{eq:contraction1}
\normso{\dot z(t)}{2}^2 +\normso{\nabla z(t)}{2}^2+\ints \normso{\nabla\dot z}{2}^2 \dt 
\\ \leq\;  C_5(r) \left(\normtso{f-\tilde{f}}{2}{2}^2  + \normtso{(a -\tilde{a}) \ddot\tpsi}{2}{2}^2\right).
\end{multline}
Next we test \eqref{eq:contractivity} with $\ddot z$ and integrate by parts in time and space to obtain
\begin{multline*}
\ints \left[ (a(s)\ddot z ,\ddot z) + b (\nabla \dot z, \nabla \ddot z) \right]\dt =\;  c^2 \ints \normso{\nabla \dot z}{2}^2 \dt -  c^2 (\nabla z (t), \nabla \dot z(t))\\
 + \ints (f - \tilde{f} - (a -\tilde{a}) \ddot \tpsi , \ddot z)\dt.
\end{multline*}
We estimate the term
\begin{align} \label{eq:c2_finalt_trick}
\begin{split}
c^2 \vert(\nabla z (t), \nabla \dot z(t)) \vert \leq\; &  \frac{1}{4\epsilon} \normso{\nabla z (t)}{2}^2  +  \epsilon \normso{\nabla \dot z (t)}{2}^2, 
\end{split}
\end{align} 
and the term
\begin{align*} 
\begin{split}
\ints \vert(f - \tilde{f} - (a -\tilde{a}) \ddot \tpsi , \ddot z)\vert \dt \leq\; & \frac{1}{4\epsilon} \normtso{f-\tilde{f}}{2}{2}^2  + \frac{1}{4\epsilon} \normtso{(a -\tilde{a}) \ddot\tpsi}{2}{2}^2 \\ 
&+ \epsilon \ints \normso{\ddot z}{2}^2  \dt 
\end{split}
\end{align*} 
The terms $\ints \normso{\nabla \dot z}{2}^2 \dt$ and $\normso{\nabla z (t)}{2}^2$ are directly estimated by \eqref{eq:contraction1}.
We choose $\epsilon$ as small as needed to absorb the terms $\normso{\nabla \dot z(t)}{2}^2$ and $\ints\normso{\ddot z}{2}^2 \dt$ by the left-hand side.
Using the fact that 
\begin{align*}
\normso{z (t)}{2}^2 \leq t^2 \normtso{\dot z}{\infty}{2}^2, \quad t \in [0,T],
\end{align*} 
we combine the resulting estimate with \eqref{eq:contraction1} to conclude that there exist $C_6=C_6(r)>0$, such that 
\begin{align*}
\begin{split}
\normx{z}^2 \leq C_6 \left(\normtso{f-\tilde{f}}{2}{2}^2  + \normtso{(a -\tilde{a}) \ddot\tpsi}{2}{2}^2  \right).
\end{split}
\end{align*}
We know that 
\begin{align*}
\begin{split}
\normtso{(a -\tilde{a}) \ddot \tpsi}{2}{2}^2 
\leq & 4 k^2 \co^2 \normtho{\dot y - \dot{\tilde{y}}}{\infty}{1}^2 \normtho{\ddot \tpsi}{2}{1}^2,
\end{split}
\end{align*}
and that
\begin{align*}
\begin{split}
\normtso{f-\tilde{f}}{2}{2}^2 =\; & 4\sigma^2 \normtso{\nabla (y-\tilde{y}) \cdot \nabla \dot y + \nabla (\dot {y} - \dot{\tilde{y}})\cdot \nabla \tilde{y}}{2}{2}^2 \\
\leq\; & 8 \sigma^2 \normtso{\nabla (y-\tilde{y})}{\infty}{2}^2 \normtso{\nabla \dot y}{2}{\infty}^2  \\& +
 8\sigma^2 \normtso{\nabla (\dot {y} - \dot{\tilde{y}})}{2}{2}^2 \normtso{\nabla \tilde{y}}{\infty}{\infty}^2.
\end{split}
\end{align*}
Together, this gives us that 
\begin{align*}
\normx{ \mathrm{A}[y]- \mathrm{A}[\tilde{y}]}^2 \lesssim \left(\normtho{\ddot \tpsi}{2}{1}^2 +
  \normtso{\nabla \dot y}{2}{\infty}^2 
+ \normtso{\nabla \tilde{y}}{\infty}{\infty}^2 \right) \normx{y-\tilde{y}}^2.
\end{align*}
Since $\tpsi\in\mathcal{W}$, we have 
$  \normtho{\ddot \tpsi}{2}{1}^2 \leq r^2,$
and we conclude that there exists $\tilde C  = \tilde C(r) >0$ (we deal with the dependence on $r$, similarly to how we did earlier in this proof for $C_4(r)$)
\begin{align*}
\normx{ \mathrm{A}[y]- \mathrm{A}[\tilde{y}]}^2 \leq  \tilde C r^2 \normx{y-\tilde{y}}^2,
\end{align*}
which shows that $\mathrm{A}[\cdot]$ is a strict contraction if 
\(
\tilde Cr^2<1.
\)

Following a similar argument to \cite[Theorem 2]{kaltenbacher2016shape}, it can be shown that the set $\mathcal{W}$ is closed in $\hat X$. Banach's fixed-point theorem then ensures the existence and uniqueness of a solution $\psi \in \mathcal {W}$ of \eqref{pb:wellposedness}. The energy estimate is obtained by using the linear estimate \eqref{eq:energynormth}, the estimates \eqref{eq:est_L2_f} and \eqref{eq:est_L1_dotf} for the source term $f$, and Gronwall's inequality. In a similar manner, the uniqueness in $W$ follows.
\end{proof}

\subsection{Shape deformations and uniform well-posedness}
We next introduce a family of admissible perturbations \(\{\Omega_d\}\)
of the reference domain $\Omega  \in C^{2,1}$ such that $\overline\Omega \subset U$, where $U$ is a fixed bounded hold-all domain in $\Rs^l$ with $l\in \{2,3\}$; see \Cref{fig:setup}. To this end, we employ the approach of \cite{murat1976controle}, 
and perturb the identity mapping with a vector field 
\[h \in \mathcal{D} = \left\{ h\in C^{2,1}\left(\overline U, \Rs^l\right) \;\vline \; h\vert_{\partial U} = 0\right\}.\]
For $h \in \mathcal{D}$ and $d\in \Rs$, such that $\absd$ is sufficiently small, we define
\[F_d = \mathrm{id} + dh.\]
Then there exists $\delta >0$, such that \(F_d(U) = U\) and that $F_d$ is a diffeomorphism for $\absd< \delta$. 
The perturbed domains and boundary are defined as
\[\Omega_d = F_d(\Omega), \qquad \bndry_d = F_d(\bndry).\]
We next employ the method of mappings to transform 
the perturbed state constraint to the reference domain $\Omega$.
We use the notation \[
\varphi_d : \Omega_d \rightarrow \Rs
\]
for quantities defined on the perturbed domain. Furthermore, we denote 
\[
\varphi^d = \varphi_d \circ F_d : \Omega \rightarrow \Rs. 
\]
If $\psi_d$ is the solution of problem \eqref{pb:wellposedness} on a deformed domain $\Omega_d$, then \[\psi^d = \psi_d \circ F_d : \Omega \rightarrow \Rs,\]
satisfies an initial-boundary value problem on the reference domain $\Omega$. To state the weak form solved by $\psi^d$, we introduce the short-hand notation 
\begin{align*}
\begin{array}{lcl}
I_d = \det DF_d = \det(\mathrm{I} + d\,\nabla h ^T),& \qquad & A_d = (DF_d)^{-T}, \\[1mm]
w_d = I_d \vert A_d n \vert, & & M_d = I_d A_d^T A_d,
\end{array}
\end{align*}
where $DF_d$ is the Jacobian of $F_d$ and $n$ denotes the outer normal unit vector to $\Omega$. 
Furthermore, we denote:
\begin{align} \label{eq:smoothdefM}
\Mostafa{\left(M_d\right)'(0)} = M = I \div h - \nabla h - (\nabla h)^T,
\end{align}
\Mostafa{where $(\cdot)'$ stands for the derivative with respect to the shape deformation variable $d$.}
The limits that define the derivative at $d=0$ exist uniformly in $\overline{U}$; see~\cite{ito2006variational}. For the upcoming analysis, we require certain known regularity properties of the transformation $F_d$ and the transformation of integral formulas. We collect these in Appendix~\ref{ap:defprop} for convenience of the reader.\\
\indent With the above notation, using the transformation of integrals, we find that the function $\psi^d = \psi_d \circ F_d$ solves
\begin{align}\label{eq:holder2}
\begin{split} 
 &\int_{\Omega} I_d (1-2k\dot\psi^d)\ddot \psi^d \phi \dx \\
 & +  \int_{\Omega} \left(c^2 M_d \nabla \psi^d \cdot \nabla \phi + b M_d\nabla \dot \psi^d \cdot \nabla \phi 
		- 2 \sigma(M_d\nabla\dot\psi^d \cdot \nabla\psi^d)\phi\right)\dx  \\
		= \; &  \int_{\bndry} w_d \left(c^2 g^d + b \dot g^d\right) \phi \ds
\end{split}
\end{align}
for all $\phi \in H^1(\Omega)$ a.e.\ in time, with initial conditions $\psi^d (0) = \psi_0^d$, $\psid^d (0) = \psi_1^d$. By generalizing the arguments of \cref{th:linearized} and \cref{th:nonlinreg}, we can prove that this problem has a unique solution as well, which remains uniformly bounded in $\|\cdot\|_{W}$. 
\begin{theorem}\label{prop:uniformwellp}
Given $T>0$, let \cref{asu:ori} hold and let $\absd$ be small enough. There exists $r_1^2=r_1^2(\Omega, T)$, independent of $d$, such that if
\begin{multline*} 
\normho{\psi_0^d}{3}^2 + \normho{\psi_1^d}{2}^2 +\normthg{g^d}{1}{3/2}^2 \\ + \normthg{g^d}{2}{-1/2}^2 + \normtlhg{\dot g^d}{\infty}{1/2}^2 \leq r_1^2,
\end{multline*}
then problem \eqref{eq:holder2} has a unique solution $\psi^d \in W$. Furthermore, the solution verifies the estimate 
\begin{align*}
\left\| \Mostafa{\psi^d}\right\|_W^2
\leq \; & C^*(\Omega, T, r_1),
\end{align*}
where the constant $C^*=C^*(\Omega, T, r_1)$ is independent of $d$. 
\end{theorem}
\begin{proof}
We can express \eqref{eq:holder2} in strong form as follows:
\begin{empheq}{align*}
	\begin{array}{rl}
	 I_d \left(1-2k\dot \psi^d \right)\ddot \psi^d -  \nabla\cdot\left( M_d (c^2 \nabla\psi^d + b \nabla\psid^d)\right)  
	- 2\sigma M_d\nabla\dot\psi^d\cdot\nabla\psi^d= 0 &\quad \mathrm{on\;} Q, \\
	M_d\left(c^2 \nabla \psi^d + b \nabla \psid^d\right)\cdot n =  w_d(c^2 g^d + b \dot g^d) &\quad \mathrm{on\;} \Sigma, \\
	\psi^d(0) = \psi^d_0,\; \dot\psi^d(0) = \psi^d_1  &\quad \mathrm{on\; \Omega},\\
	\end{array}
\end{empheq}
so we can exploit the results of \cite[Chapter 6]{murat1976controle}, where a steady-state counterpart of this problem is considered. 

Note that $M_d(x)$ is uniformly positive for $x \in \Omega$ according to, e.g, \cite{ito2006variational} and \cite[Proposition 2.12]{sturm2015shape}. Further, $\|M_d\|_{C^{1,1}(\overline{\Omega})}$ and $\|w_d\|_{C^{1,1}(\partial \Omega)}$ are uniformly bounded for $\absd$ small enough thanks to the smoothness properties of the transformation $F_d$; see Appendix~\ref{ap:defprop}. Thus, one can readily show that the energy estimates for the above problem (that can be established analogously to those in \cref{th:linearized}) can be made independent of $d$ for $\absd$ small enough.

\indent By generalizing the linear and fixed-point arguments of \cref{sec:wellp} (noting that elliptic regularity \cite[Theorem 2.5.1.1]{grisvard2011elliptic} can again be used thanks to the aforementioned regularity of the shape deformation quantities), 
one can show that \eqref{eq:holder2} has a unique solution in $W$ under a smallness bound on $r_1$, which is independent of $d$.
\end{proof}
\section{H\"older-continuity with respect to shape} \label{sec:Holder_cont}
In this section, we establish H\"older-continuity of the solution to the nonlinear state problem with respect to shape deformations. This result is crucial in the rigorous justification of the shape derivatives.
\begin{theorem} \label{th:Holder} 
Let $T >0$  
and assume 
$\psi_0^d$, $\psi_1^d$ and $g^d$ verify the assumptions of \cref{prop:uniformwellp} uniformly in $d$ for $\absd$ small enough. 
Furthermore, assume H\"older-continuity of data with respect to $d$ in the following sense:
\begin{enumerate}[label=\roman*)]
 \item Initial data 
\[\lim_{d\rightarrow 0} \frac{1}{d} \normho{\psi_0^d - \psi_0 }{1}^2  = 0, \qquad \lim_{d\rightarrow 0} \frac{1}{d} \normso{\psi_1^d - \psi_1 }{2}^2  = 0;\]
  \item Boundary data 
\[\lim_{d\rightarrow 0} \frac{1}{d} \normthg{g^d - g }{1}{-1/2}^2  = 0.\]
\end{enumerate} 
Then the solution of the Kuznetsov problem \eqref{pb:wellposedness} is H\"older-continuous with respect to shape deformations in the sense of
\[\lim_{d\rightarrow 0} \frac{1}{d} \left\| \psi^d - \psi \right\|^2_X  = 0,\]
where $ X= W^{1,\infty}(0,T; L^2(\Omega) \cap H^1(0,T; H^1(\Omega))$. 
\end{theorem}
\begin{proof}
Let $z^d$ stand for the difference quotient:
\[z^d = \frac{1}{\sqrt{d}} (\psi^d-\psi^0) = \frac{1}{\sqrt{d}} (\psi^d-\psi).\] 
We recall that $\psi$ and $\psi^d$ satisfy \eqref{eq:holder2}. By subtracting their respective weak forms, we obtain the following equation for $z^d$:
 \begin{align*} 
 \begin{split}
 &\left((1-2k\dot\psi)\ddot z^d, \phi\right) + \left(c^2 \nabla z^d + b \nabla \dot z^d , \nabla \phi\right)\\
 =&\,  \left( 2 k \ddot \psi^d \dot z^d, \phi\right)   -\frac{I_d-1}{\sqrt{d}} \left((1-2k\dot\psi^d)\ddot\psi^d, \phi\right)
 \\&- \left( \frac{M_d-I}{\sqrt{d}}(c^2 \nabla\psi^d + b \nabla \dot\psi^d),\nabla \phi\right) 
 + 2 \sigma \left(\frac{M_d-I}{\sqrt{d}}\nabla \dot\psi^d \cdot \nabla\psi^d, \phi\right) \\
&\qquad -2 \sigma(\nabla\dot\psi\cdot \nabla z^d, \phi)  + 2 \sigma (\nabla\psi \cdot \nabla \dot z^d,\phi)
+ \int_{\bndry} \frac{w_d-1}{\sqrt{d}} \left(c^2 g^d + b \dot g^d\right) \phi \ds \\
&\qquad + \int_{\bndry} \left(c^2  \frac{g^d-g}{\sqrt{d}} + b \frac{\dot g^d-\dot g}{\sqrt{d}}\right) \phi \ds,
 \end{split}
 \end{align*}
with initial conditions
\[
z^d(0) = \frac{1}{\sqrt{d}} (\psi_0^d-\psi_0) ,\qquad \dot z^d(0) = \frac{1}{\sqrt{d}} (\psi_1^d-\psi_1).
\] 
We next test this difference equation with $\phi= \dot z^d \in L^2(0,T; H^1(\Omega))$ and integrate over $(0, t)$. We again use the notation $a = 1-2k\psid$ (which implies $\dot a = - 2k\psidd$), and relation \eqref{eq:energy_id_1},
we then have after rearranging the terms
\begin{align} \label{z_est}
\begin{split}
&\frac{1}{4}\normso{\zd(t)}{2}^2 +\frac{c^2}{2} \normso{\nabla \z(t)}{2}^2 + b \ints \normso{\nabla \zd}{2}^2 \dt 
\\\leq \;  & \frac{3}{4}\normso{\zd(0)}{2}^2 +\frac{c^2}{2} \normso{\nabla \z(0)}{2}^2 - \frac{1}{2}\ints(\dot a\zd,\zd) \dt \\& -  \ints \left(\frac{I_d-1}{\sqrt{d}} a^d \psidd^d ,\zd\right)\dt - 
  \ints \left[ \left( \frac{M_d-I}{\sqrt{d}} (c^2 \npsi + b \npsid), \nzd \right) \right.  
  \\& + \left. 2\sigma\left(\frac{M_d-I}{\sqrt{d}}\nabla \dot\psi^d \cdot \nabla\psi^d, \zd\right) \right] \dt  
  - 2\sigma \ints  (\npsid \cdot \nz, \zd) \dt \\ & + 2\sigma \ints (\npsi \cdot \nzd, \zd) \dt 
 + \ints \intg \frac{w_d-1}{\sqrt d} (c^2 g^d + b \dot g ^d) \zd \ds\dt 
 \\ &+  \ints \intg \left( c^2  \frac{g^d-g}{\sqrt{d}} + b \frac{\dot g^d-\dot g}{\sqrt{d}}\right) \zd \ds \dt. 
\end{split}
\end{align} 
By relying on H\"older's inequality and Ehrling's inequality \eqref{eq:usefulineq}, we obtain 
\begin{align*} 
\ints \vert (\dot a \zd, \zd) \vert \dt \leq\;& \normtso{\dot a}{\infty}{2} \ints \normso{\zd}{4}^2 \dt \\
\leq \;& \normtso{\dot a}{\infty}{2}\ints \left(\ce(\epsilon_1)\normso{\zd}{2}^2 + \epsilon_1 \normso{\nabla \zd}{2}^2\right) \dt
\end{align*}
for all $\epsilon_1>0$.
Furthermore, we have
\begin{align*}
\begin{split}
\ints \left(\frac{I_d-1}{\sqrt{d}} a^d\ddot\psi^d, \zd\right) \dt 
\leq&\, \begin{multlined}[t] \frac{1}{2}\frac{\normso{I_d-1}{\infty}^2}{d} \|a^d \|^2_{L^\infty(L^\infty(\Omega))} \ints \normso{\psidd^d}{2}^2 \dt 
\\+ \frac{1}{2}\ints \normso{\zd}{2}^2\dt. \end{multlined}
\end{split}
\end{align*}
We next use the differentiability of the shape deformation $F_d$ and of the quantities $I_d$, $M_d$ and $w_d$ (see \Cref{ap:defprop}), as well as the uniform boundedness of $\psi^d$ in the norm $\|\cdot\|_W$ to further simplify the expression as:
\begin{align*}
\begin{split}
\ints \left(\frac{I_d-1}{\sqrt{d}} a^d\ddot\psi^d, \zd\right) \dt
\leq&\, \frac{1}{2}\ints \normso{\zd}{2}^2\dt +\bigO(\absd).
\end{split}
\end{align*}
We treat similarly the other volume terms and estimate them as follows
\begin{align*}
	\begin{split}
		\ints  \left( \frac{M_d-I}{\sqrt{d}} (c^2 \npsi + b \npsid)\right. \left.\vphantom{\frac{I}{d}}, \nzd \right) \dt 
		\leq\, \frac{\epsilon_1}{2}\ints \|\nzd\|^2_{L^2(\Omega)} \dt + \bigO(\absd),
	\end{split}
\end{align*}
while, 
\begin{equation*}
	\begin{aligned}
		\begin{multlined}
			\ints  \left( \frac{M_d-I}{\sqrt{d}} \npsid^d \cdot\npsi^d, \zd \right) \dt
			\leq  \frac{1}{2}\ints \|\zd\|^2_{L^2(\Omega)} \dt  + \bigO(\absd),
		\end{multlined}
	\end{aligned}
\end{equation*}
because $\psi^d$ and $\dot{\psi^d}$ are uniformly bounded in $W$ thanks to \cref{prop:uniformwellp}. We also get 
\begin{equation*}
	\begin{aligned}
		\begin{multlined}[t]
			\ints  \left( \npsi^d \cdot\nzd, \zd \right) \dt \\
			\leq  \frac{1}{2} \normtso{\npsi^d}{\infty}{\infty} \ints \left(\epsilon_1 \normso{\nzd}{2}^2 + \inv{\epsilon_1}\normso{\zd}{2}^2 \right)\dt.
		\end{multlined}
	\end{aligned}
\end{equation*}
Finally, we have, for all $\epsilon>0$
\begin{equation*}
	\begin{aligned}
		\begin{multlined}[t]
			\ints  \left( \npsid^d \cdot\nz, \zd \right) \dt\\
			\leq  \frac{\epsilon}{2} \sup_{s\in (0,t)} \normso{\nz(s)}{2}^2 + \frac{1}{2\epsilon} \normtso{\npsid^d}{2}{\infty}^2 \ints \normso{\zd}{2}^2 \dt.
		\end{multlined}
	\end{aligned}
\end{equation*}
We estimate the boundary terms as follows:
\begin{equation*}
	\begin{aligned}
		\begin{multlined}[c]
			\ints  \intg \frac{w_d-1}{\sqrt d} (c^2 g + b  \dot g)  \zd \ds \dt 
			 \leq\; \epsilon_1\ctr^2  \ints \normho{\zd}{1}^2 \dt + \bigO(\absd),
		\end{multlined}
	\end{aligned}
\end{equation*}
where $ \bigO(\absd)$ term arises from the shape differentiability of $w_d$. Similarly,
\begin{equation*}
\begin{aligned}
\begin{multlined}[t]
\ints \intg \left( c^2  \frac{g^d-g}{\sqrt{d}}\right.  + \left. b \frac{\dot g^d-\dot g}{\sqrt{d}}\right) \zd \ds \dt
\\ \leq\;  \frac{1}{4 \epsilon_1}\max{(c^4,b^2)}\frac{\normthg{g^d-g}{1}{-1/2}^2}{d} 
 + \epsilon_1\ctr^2  \ints \normho{\zd}{1}^2 \dt.
\end{multlined}
\end{aligned}
\end{equation*}
We can then choose $\epsilon_1$ small enough so that 
$$ \epsilon_1 \left(\ctr^2 + \frac{1}{2}+ \frac{1}{2} \normtso{\npsid}{\infty}{\infty} + \normtso{\dot a}{\infty}{2}\right) < b$$
and the terms with 
$\ints \normso{\nzd}{2}^2\dt$ on the right-hand side can be absorbed by the left-hand side in \eqref{z_est}. We notice also that the terms with initial conditions and boundary data are $\smallO(1)$ (i.e, they tend to 0 when $d\rightarrow 0$ ). We then infer that 
\begin{align*}
\begin{split}
&
\normso{\zd(t)}{2}^2 + \normso{\nz(t)}{2}^2 + \ints \normso{\nzd}{2}^2\dt  \\ \lesssim &  \ints \normso{\zd}{2}^2 \dt +  \ints \normso{\nz}{2} ^2 \dt 
 + \epsilon \sup_{s\in (0,t)} \normso{\nz(s)}{2}^2 + \bigO(\absd) +\smallO(1).
\end{split}
\end{align*}
Noticing that each of the terms on the left-hand side is bounded by the right-hand side independently, we take the supremum for $\tau \in(0,t)$ to get the following bound:
\begin{align*}
\begin{split}
&
\sup_{\tau \in(0,t)} \normso{\zd(\tau)}{2}^2 + \sup_{\tau \in(0,t)} \normso{\nz(\tau)}{2}^2 + \ints \normso{\nzd}{2}^2\dt  
\\ \lesssim &  \ints \normso{\zd}{2}^2 \dt +  \ints \normso{\nz}{2} ^2 \dt 
 + \epsilon \sup_{s\in (0,t)} \normso{\nz(s)}{2}^2 + \smallO(1).
\end{split}
\end{align*}
We then fix $\epsilon>0$ sufficiently small and use that 
\begin{align*}
\normso{\z (t)}{2} = \normso{\z(0) + \ints \zd (s) \dt}{2} \leq\ t \sup_{\tau \in (0,t)} 
\normso{\zd(\tau)}{2},
\end{align*}
together with Gronwall's inequality to get
\begin{align*}
\begin{split}
\normso{\z(t)}{2}^2 + \normso{\zd(t)}{2}^2 + \normso{\nz(t)}{2}^2 + \ints \normso{\nzd}{2}^2\dt \lesssim \ \smallO(1).
\end{split}
\end{align*} 
This final inequality concludes our proof.
\end{proof}
\section{Analysis of a general adjoint problem}\label{sec:adjp}
Different cost functions are of practical interest for the optimization problem at hand. As announced, we first consider an $L^2(L^2)$ tracking-type objective. That is, we wish the solution $\psi$ of the state problem to match a 
desired potential field $\psi_{D}\in L^2(0,T;L^2(\Omega))$ in the focal region $D_S$ within a certain time interval $(t_0, t_1)$. Such a cost functional can be expressed by
\begin{align}\label{eq:costf1}
	J(\Omega) = \frac{1}{2}\int_0^T \int_{\Omega} \left(\psi - \psi_D\right)^2 \chi_D \dxs= \int_0^T \int_{\Omega} j(\psi)\dxs,
\end{align}
where $\chi_D$ is the indicator function of \(D = D_S \times (t_0,t_1)\).
The adjoint problem is then formally given by
\begin{empheq}{align}\label{eq:sysAdjKuz1}
	\begin{array}{rl}
	\pdfrac{}{t}\left((1-2k {\dot\psi})\dot p\right) - c^2 \laplace p 
			 + b\laplace \dot p - 2 \sigma \nabla \cdot\left(\dot p \nabla{\psi}\right)= ( \psi - \psi_D)\chi_D & \ \mathrm{on\;} Q,\\
	c^2\pdfrac{p}{n}-b\pdfrac{\dot p}{n}+ 2 \sigma g\dot p =0 &\ \mathrm{on\;} \Sigma,\\
	p(T) = \dot p(T) = 0&\ \mathrm{on\; \Omega},
	\end{array}
\end{empheq}
which in weak form yields
\begin{equation} \label{eq:weak_adj_1}
\begin{aligned}
\begin{multlined}[t]\int_0^T\int_{\Omega}\pdfrac{}{t}\left((1-2k {\dot\psi})\dot p\right) \phi \dx \dt  
\\+ \int_0^T\int_{\Omega}(c^2 \nabla p - b\nabla\dot p + 2 \sigma\dot p \nabla{\psi}) \cdot\nabla \phi\dx\dt
= \int_0^T\int_{\Omega}f \phi\dx\dt, \end{multlined}
\end{aligned}
\end{equation}
for all $\phi \in L^2(0,T;H^1(\Omega))$ with $p(T) = \dot p(T) = 0$ and $f=(\psi-\psi_{D})\chi_D$.

\indent For the sake of considering other objective functionals, we analyze the adjoint problem in a more general setting, where we allow for non-zero data at final time:
	\[
	p(T)=p_0, \qquad \dot{p}(T)=p_1
	\]
and a general right-hand side $f \in L^2(0,T; L^2(\Omega))$. 
\begin{theorem}\label{th:adjreg}
Under the assumptions of \cref*{th:nonlinreg} and given $f \in L^2(0,T; L^2(\Omega))$, the adjoint problem
\eqref{eq:sysAdjKuz1} has a unique solution $p$ in 
\begin{equation}\label{eq:Pdef}
\begin{aligned}
\mathcal{P} = \left\{ p \in L^{\infty}(0,T;H^2(\Omega))\right.:&\, \dot p \in L^2(0,T;H^2(\Omega))\cap L^{\infty}(0,T;H^1(\Omega)),
\\ &\,\left. \ddot p \in L^2(0,T,L^2(\Omega)) \right\}.
\end{aligned}
\end{equation}
Furthermore, the solution satisfies the estimate
\begin{align*}
\| p\|_{\mathcal{P}}^2 \lesssim \normho{p_0}{2}^2 + \normho{p_1}{1}^2 + \normtlhg{g}{2}{3/2}^2 +\intt\|f\|_{L^2(\Omega)}^2\dt.
\end{align*} 
\end{theorem}
\begin{proof}
	To facilitate the analysis,
	we first reverse the time direction. We then see that 
	$\tilde{p}(t) = p(T-t)$
	satisfies
	\begin{equation} \label{eq:adjreg} 
	\begin{aligned}
	\begin{multlined}[t]\int_0^T\int_{\Omega}\pdfrac{}{t}\left((1+2k {\dot\tpsi})\dot \tp\right) \phi \dx \dt  
	\\+ \int_0^T\int_{\Omega}(c^2 \nabla \tp + b\nabla\dot \tp - \sigma \dot \tp \nabla{\tpsi}) \cdot\nabla \phi\dx\dt = \; 
	\int_0^T\int_{\Omega}\tilde{f} \phi \dx\dt, \end{multlined}
	\end{aligned}
	\end{equation}
	for all $\phi \in L^2(0,T;H^1(\Omega))$ with $\tp(0) = p_0 $, $\dot \tp(0) = p_1$ and  $\tilde{f}=(\tpsi-\tpsi_{D})\chi_D \in L^2(0,T; L^2(\Omega))$. \\
\indent The existence of weak solutions can be obtained rigorously using standard Galerkin approximations. As before, we will only present the derivation of an a priori estimate here.
We derive the energy estimate by testing the (semi-discrete) weak form 
\eqref{eq:adjreg} with $\dot \tp$ and $\ddot \tp$. We then combine it with an elliptic regularity argument to enhance the regularity of the solution.\\[2mm]
\noindent\textbf{Lower-order energy estimate.}
We test \eqref{eq:adjreg} with $\dot\tp$ 

\begin{align*}
\begin{split}
\ints \left((1+2k {\dot\tpsi}) \ddot\tp + 2k\ddot\tpsi \dot\tp, \dot \tp\right) \dt   
+ \ints (c^2 \nabla \tp + b\nabla\dot \tp - \sigma \dot \tp \nabla{\tpsi},\nabla \dot\tp)\dt = 
\ints(\tilde{f}, \dot \tp)\dt.
\end{split}
\end{align*}
 
We denote by $\tilde a =1 +2k \dot\tpsi$. $\psi \in \mathcal{W}$ implies that $\tpsi \in \mathcal{W}$; we therefore have that 
$\frac{1}{2}<\tilde a< \frac{3}{2}$ almost everywhere on $\Omega$ for all $t \in [0,T]$. Knowing that 
$$ \tilde a \ddot\tp \dot\tp = \frac{1}{2} \diffrac{}{t}(\tilde a \dot\tp^2) - \frac{1}{2}(\dot{\tilde a} \dot\tp^2),$$
we infer  
\begin{multline*}
\frac{1}{4} \normso{\dot \tp(t)}{2}^2  + \frac{c^2}{2} \normso{\nabla \tp(t)}{2}^2 
+ b \ints \normso {\nabla \dot \tp}{2}^2\dt\\
\leq \frac{3}{4} \normso{\dot \tp(0)}{2}^2  + \frac{c^2}{2} \normso{\nabla \tp(0)}{2}^2 
+ \sigma \ints\int_{\Omega}  \dot \tp \nabla{\psi}\cdot\nabla \dot\tp\dx\dt \\
- \frac{1}{2}\ints (\tilde{\dot a} \dot \tp,\dot\tp)  \dt + \frac{1}{2}\|\tilde{f}\|_{L^2(L^2(\Omega))}^2 + \frac{1}{2}\ints\normso{\dot \tp}{2}^2\dt.
\end{multline*}
We estimate the term
\begin{equation*}
	\begin{aligned}
		\begin{multlined}[t]
			\ints\int_{\Omega}  \dot \tp \nabla{\psi}\cdot\nabla \dot\tp\dx\dt \leq  \normtso{\nabla\tpsi}{\infty}{\infty}^2 \frac{1}{4\epsilon} \ints \normso{\dot\tp}{2}^2\dt + \epsilon \ints \normso{\nabla\dot\tp}{2}^2\dt.
		\end{multlined}
	\end{aligned}
\end{equation*}
\\
\noindent Using Ehrling's inequality \eqref{eq:usefulineq}, we estimate, for some arbitrary $\epsilon>0$, the term 
\begin{align*}
	\begin{split}
		\ints \vert(\dot{\tilde a} \dot \tp,\dot\tp)\vert  \dt \leq \normtso{\dot{\tilde a}}{\infty}{2} \ints \left(\ce(\epsilon)\normso{\dot\tp}{2}^2 + \epsilon \normso{\nabla \dot \tp}{2}^2\right) \dt.
	\end{split}
\end{align*}
For $\epsilon$ small enough this yields the following estimate: 
\begin{multline}\label{eq:adjregest1}
\normso{\dot \tp(t)}{2}^2  + \normso{\nabla \tp(t)}{2}^2 
+  \ints \normso {\nabla \dot \tp}{2}^2\dt\\
\lesssim \;  \normso{\dot \tp(0)}{2}^2  +  \normso{\nabla \tp(0)}{2}^2 
+ \|\tilde{f}\|_{L^2(L^2(\Omega))}^2.
\end{multline}
Next we test with $\ddot \tp$
\begin{align} \label{eq:adjreg2}
\begin{split}
\ints \left(\tilde{a} \ddot\tp + \dot{\tilde a} \dot\tp, \ddot \tp\right) \dt   
+ \ints (c^2 \nabla \tp + b\nabla\dot \tp , \nabla\ddot\tp)\dt 
- 2 \sigma \ints(\dot \tp  \nabla \tpsi,\nabla \ddot\tp)\dt\\ = 
\ints(\tilde{f}, \ddot \tp)\dt.
\end{split}
\end{align}
We integrate by parts in time the terms 
\begin{align*}
\begin{split}
c^2\ints ( \nabla \tp , \nabla\ddot\tp)\dt = - c^2 \ints ( \nabla \dot\tp , \nabla\dot\tp)\dt + c^2 \left(\nabla \tp(t) , \nabla\dot\tp(t)\right) \\
- 2\sigma \ints(\dot \tp  \nabla \tpsi,\nabla \ddot\tp)\dt = 2\sigma \ints(\ddot \tp  \nabla \tpsi + \dot \tp  \nabla\dot \tpsi ,\nabla \dot\tp)\dt
- 2\sigma \left(\dot \tp(t)  \nabla \tpsi(t),\nabla \dot\tp(t)\right).
\end{split}
\end{align*}
With this, we can rewrite \eqref{eq:adjreg2} as
\begin{multline*}
\ints \left(\tilde a \ddot\tp, \ddot \tp\right) \dt      
+ b \ints (\nabla\dot \tp , \nabla\ddot\tp)\dt 
 =  - \ints (\dot{\tilde a} \dot\tp, \ddot \tp) \dt  + c^2 \ints (\nabla \dot \tp , \nabla\dot\tp)\dt
\\- c^2 \left(\nabla \tp(t) , \nabla\dot\tp(t)\right) + 2\sigma \ints(\ddot \tp  \nabla \tpsi + \dot \tp  \nabla\dot \tpsi ,\nabla \dot\tp)\dt 
+ 2\sigma \left(\dot \tp(t)  \nabla \tpsi(t),\nabla \dot\tp(t)\right) \\
+ \ints(\tilde{f}, \ddot \tp)\dt.
\end{multline*}
Next we estimate the right-hand side terms starting with the term 
\begin{align*}
\begin{split}
\ints \vert(\dot{\tilde a} \dot\tp, \ddot \tp)\vert \dt
 \leq & \co \normtho{\dot{\tilde a}}{2}{1} \sup_{\tau \in (0,t)}\normso{\dot \tp (\tau)}{4} \left(\ints \normso{\ddot \tp}{2}^{2}\dt\right)^{1/2}  \\
\leq & \frac{\co^2 }{4 \epsilon_1} \normtho{\dot{\tilde a}}{2}{1}^2 \sup_{\tau \in (0,t)}\normso{\dot \tp (\tau)}{4}^2 
+ \epsilon_1 \ints \normso{\ddot \tp}{2}^{2}\dt. \\
\end{split}
\end{align*}
Using Ehrling's inequality \eqref{eq:usefulineq}, we obtain for all $\epsilon_2>0$
$$\normso{\dot \tp (\tau)}{4}^2 \leq \ce(\epsilon_2)\normso{\dot \tp (\tau)}{2}^2 + \epsilon_2\normso{\nabla \dot \tp (\tau)}{2}^2,$$
we can first choose $\epsilon_1$ small enough to absorb the term $\ints \normso{\ddot \tp}{2}^{2} \dt $, 
then $\epsilon_2=\epsilon_2(\epsilon_1)>0$ afterwards small enough to absorb the term 
$\normso{\nabla \dot \tp (\tau)}{2}^2$. 

Usual computations then lead to the bound
\begin{multline}\label{eq:adjregest2}
\ints \normso{\ddot \tp}{2}^2 \dt      
+ \normso{\nabla\dot \tp (t)}{2}^2 
\lesssim \; \normso{\nabla\dot \tp (0)}{2}^2 + \sup_{\tau \in (0,t)}\normso{\dot \tp (\tau)}{2}^2  
\\+ \sup_{\tau \in (0,t)}\normso{\nabla \tp(t)}{2}^2 + \|\tilde{f}\|_{L^2(L^2(\Omega))}^2.
\end{multline}
Estimates \eqref{eq:adjregest1} and \eqref{eq:adjregest2} yield
\begin{multline*}
 \normso{\dot \tp(t)}{2}^2  + \normso{\nabla \tp(t)}{2}^2       
+ \normso{\nabla\dot \tp (t)}{2}^2 + \ints \normso{\ddot \tp}{2}^2 \dt
\\
\lesssim \normso{\nabla\dot \tp (0)}{2}^2 + \normso{\dot \tp(0)}{2}^2  +  \normso{\nabla \tp(0)}{2}^2 + \|\tilde{f}\|_{L^2(L^2(\Omega))}^2.
\end{multline*}

\noindent\textbf{Bootstrap argument.}
From the above-analysis we can infer that 
$$\tp \in W^{1,\infty}(0,T;H^1(\Omega))\cap H^2(0,T,L^2(\Omega)).$$
We present the next arguments for $l = 3$; the case
$l = 2$ can be treated analogously. We introduce the new variable $\dbtilde p = c^2 \tp + b \dot \tp,$ which verifies 
\begin{empheq}{align*}
	\begin{array}{rl}
	 - \laplace \dbtilde p + \dbtilde p = - \pdfrac{}{t}\left((1+2k {\dot\tpsi})\dot \tp\right) - 2\sigma\nabla \cdot\left(\dot \tp \nabla{\psi}\right) + \tilde{f}  +\dbtilde p& \quad \mathrm{on\;} Q,\\
	\pdfrac{\dbtilde p}{n} = 2\sigma g\dot \tp  &\quad \mathrm{on\;} \Sigma.\\
	\end{array}
\end{empheq}
It is easy to see that $- \laplace \dbtilde p + \dbtilde p \in L^2(\Omega)$ a.e.\ in time. 
Additionally, if $\Omega$ is a bounded subset of $\Rs^3$, then $\bndry$ is a compact 2-manifold. One can use properties of pointwise multiplication on fractional Sobolev spaces on $n$-manifolds to show that 
since $g \in H^{3/2}(\bndry)$ and $\dot \tp \in  H^{1/2}(\bndry)$ a.e. in time, then $2\sigma g\dot \tp \in H^{1/2}(\bndry)$ a.e.\ in time; see 
\cite[Corollary 3]{holst2009rough}. We then use elliptic regularity to conclude that $\dbtilde p \in L^2(0,T;H^2(\Omega))$ 
(as done in \cite[Theorem 4.3]{KALTENBACHER20191595}, for example). 

\indent Knowing that $\tp$ solves the ODE: $\dbtilde p = c^2 \tp + b \dot \tp,$ one may write
\begin{align} \label{est_tp}
\tp = e^{-c^2t/b} \left(p_0 + \ints e^{c^2s/b} \frac{1}{b} \dbtilde p(s) \dt\right).
\end{align}
Analogously to the proof of \cref{th:linearized}, it then follows that $\tp \in L^\infty(0,T, H^2(\Omega))$ and subsequently that $\dot \tp \in L^2(0,T, H^2(\Omega))$. The energy estimate follows from \eqref{est_tp} and the elliptic regularity estimate for $\dbtilde p$. 
\end{proof}

\section{Shape derivative}\label{sec:shapeanalysis}
\Mostafa{
Given cost functional $J$, we consider the following shape optimization problem 
\begin{align*}
\left\{ 
\begin{array}{ll}
\displaystyle \min_{(\psi, \Omega) \in W \times \mathcal{O}_{ad} }& J(\psi,\Omega), \quad  \\
\textrm{s.t.} & \psi \textrm{ solves \eqref{pb:wellposedness} on } \Omega 
\end{array}
\right.
\end{align*}
with $W$ defined in \eqref{def_W}, and $\mathcal{O}_{ad}$ is the set of admissible domains:
\[
\mathcal{O}_{ad} \subset \{
\Omega \textrm{ is a domain such that }\Omega \in C^{2,1}, \ \overline\Omega \subset U
\}.
\] 
}
Our previous analysis allows us to now focus on deriving the shape sensitivity $\diff J (\Omega) h$ of a cost functional $J$ in the direction of a smooth enough vector field $h \in \mathcal{D}$. We recall that the Eulerian derivative of $J$ in the direction of the vector field $h$ is defined as
\begin{align}\label{eq:eulDer}
	\diff J(\Omega)h = \lim_{d\rightarrow 0}\frac{1}{d}\left(J(\psi_d, \Omega_d) - (J(\psi, \Omega)\right),
\end{align}
where $\psi$ and $\psi_d$ satisfy the PDE on the original and on the perturbed domain, respectively. The difficulty ensuing from the difference quotient in \eqref{eq:eulDer} involving functions defined on different domains is overcome by using the method mappings, discussed in~\cref{sec:wellp}.\\
\indent The first cost functional of interest, given in \eqref{eq:costf1}, can be written on a deformed domain  \[J (\Omega_d) = \frac{1}{2}\int_0^T \int_{\Omega_d} (\psi_d - \psi_{D_d})^2\chi_{D_d} \dxs=\frac{1}{2}\int_0^T \int_{\Omega} (\psi^d - \psi^d_D)^2\chi_D^d I_d \dxs.\]
Motivated by application purposes and to simplify the analysis, we assume that the focal region is compactly contained in $\Omega$; i.e,  $\overline {D_S} \subsetneq \Omega$. Thus
we impose $\supp h \cap \overline {D_S} = \emptyset \implies \chi_D^d = \chi_D$, and have 
\[J (\Omega_d) = \frac{1}{2}\int_0^T \int_{\Omega} (\psi^d - \psi^d_D)^2\chi_D I_d \dx.\]
\noindent To simplify the presentation and in agreement with HIFU applications, we assume going forward that the potential field is at rest at $t=0$ for all small deformations; that is,
$\psi^d(0) = \dot \psi^d(0) =0$,
for all $\absd$ small.

We generalize in what follows a useful identity \cite[Lemma 6]{kaltenbacher2016shape} to accommodate the term $\nabla\psi\cdot\nabla\dot\psi$ which will be needed in the computation of the shape derivative. 

\begin{lemma}\label{lem:trans2}
	For $a \in H^1(\Omega)$ and $u, v \in H^2(\Omega)$ the following identity holds:
	\begin{align*}
		& \int_{\Omega} a M \nabla u \cdot \nabla v  \\ 
		= & \int_{\Omega} \left(\nabla \cdot (a\nabla u) (h\cdot \nabla v)+\nabla\cdot (a\nabla v) (h\cdot \nabla u) 
		- (\nabla u \cdot \nabla v) (h\cdot\nabla a) \right)\\
		& -\int_{\partial\Omega} a \left(\pdfrac{u}{n}(h\cdot \nabla v) + \pdfrac{v}{n}(h\cdot \nabla u)\right) + \int_{\partial \Omega} a\nabla u \cdot \nabla v (h\cdot n),
	\end{align*}
	where $M = I\div h - (\nabla h)^T -(\nabla h)$, as defined in \eqref{eq:smoothdefM}.
\end{lemma}
\begin{proof}
The proof in the case $a = 1$ follows by~\cite[Lemma 5]{abda2013dirichlet}. The general case $a \in H^1(\Omega)$ follows by a straightforward extension of the arguments there.
\end{proof}
\noindent We now have all the tools to rigorously compute the shape derivative.
\begin{theorem}\label{th:derth1}
Let the assumptions made in \cref{prop:uniformwellp,th:Holder} hold and let $\psi$
and $p$ be the solutions of the state \eqref{pb:wellposedness} and adjoint \eqref{eq:sysAdjKuz1} problems, respectively. Further, let the assumptions made in this section hold. 
Then the shape derivative for cost function $J$, defined in \eqref{eq:costf1}, exists in the direction of any $h \in \mathcal{D}$ and is given by
\begin{multline}\label{eq:diffexpr}
\diff J(\Omega) h =\;  \int_0^T \int_{\bndry}(\pdfrac{}{n} ((c^2 g+b \dot g) p) + (c^2 g + b \dot g)p \kappa) (h\cdot n) \ds\dt \\
- \int_0^T\int_{\bndry} \left((1-2k\dot\psi)\ddot \psi p + c^2 \nabla p \cdot \nabla \psi + b \nabla p \cdot \nabla \dot \psi - 2 \sigma p\nabla \psi \cdot \nabla \dot\psi \right) (h\cdot n) \ds\dt,
\end{multline}
where $\kappa$ stands for the mean curvature of $\bndry$.
\end{theorem}

\begin{proof}
Inspired by the rearrangement technique of \cite{kaltenbacher2016shape,ito2008variational}, we write:
\begin{align*}
\begin{split}
J(\Omega_d)-J(\Omega) = & \,\frac{1}{2}\int_0^T\int_{\Omega} (I_d (\psi^d - \psi^d_D)^2 \chi_D-(\psi - \psi_D)^2 \chi_D)\dxs \\
 = & \, \begin{multlined}[t]\frac{1}{2}\int_0^T\int_{\Omega} \left( (I_d -1)(\psi^d - \psi^d_D)^2 + (\psi^d - \psi)^2 + \cancel{(\psi - \psi_D)^2} \right. \\
 + (\psi_D- \psi^d_D)^2 - \cancel{(\psi - \psi_D)^2} + 2 (\psi^d - \psi) (\psi - \psi_D) \\ 
\left. +2 (\psi^d - \psi)(\psi_D- \psi^d_D)  +2 (\psi - \psi_D)(\psi_D- \psi^d_D) \right)\chi_D\dxs.\end{multlined}
\end{split}
\end{align*}
We use \cref{th:Holder} together with the fact that $d \mapsto \psi_D^d$ is differentiable at $d = 0$ to arrive
at
\begin{align*}
\diff J(\Omega) h = &\int_0^T \int_{\Omega} \left(j(\psi) \div h - j'(\psi)\nabla \psi_D\cdot h\right) \chi_D \dx\dt 
\\ & + \lim_{d\rightarrow0} \frac{1}{d} \int_0^T\int_{\Omega} j'(\psi) (\psi^d-\psi) \dx\dt,
\end{align*}
with $j'(\psi)=(\psi-\psi_{D})\chi_D$. However, since $\supp h \cap \overline {D_S} = \emptyset$, it reduces to
\begin{align*}
\diff J(\Omega) h = & \lim_{d\rightarrow0} \frac{1}{d} \int_0^T\int_{\Omega} j'(\psi) (\psi^d-\psi) \dx\dt.
\end{align*}
We notice that the integral above resembles the right-hand side of the adjoint problem \eqref{eq:weak_adj_1}. Thus testing the adjoint problem with $\psi^d-\psi \in L^2(0,T; H^1(\Omega))$ yields
\begin{align*}
\begin{split}
\int_0^T \int_{\Omega} j'(\psi) (\psi^d-\psi) \dxs = & \int_0^T\int_{\Omega}\pdfrac{}{t}\left((1-2k {\dot\psi})\dot p\right) (\psi^d-\psi) \dxs  
\\&+ \int_0^T\int_{\Omega}(c^2 \nabla p - b\nabla\dot p  + 2 \sigma \dot p \nabla{\psi}) \cdot\nabla (\psi^d-\psi)\dxs.
\end{split}
\end{align*}
\noindent Integrating twice by parts in the first term on the right yields 
\begin{align}\label{eq:adj1dereq}
\begin{split}
\int_0^T \int_{\Omega} j'(\psi) (\psi^d-\psi) \dx\dt = &  \int_0^T\int_{\Omega}\pdfrac{}{t}\left((1-2k {\dot\psi}) (\dot\psi^d-\dot\psi)\right) p \dx \dt
\\&+ \int_0^T\int_{\Omega}(c^2 \nabla p - b\nabla\dot p  + 2 \sigma \dot p \nabla{\psi}) \cdot\nabla (\psi^d-\psi)\dx \dt.
\end{split}
\end{align}
We note that the following equality holds: 
\begin{align*}
\int_0^T\int_{\Omega} \dot p \nabla{\psi} \cdot\nabla (\psi^d-\psi)\dx 
=- \int_0^T\int_{\Omega}p \left(\nabla{\dot\psi} \cdot\nabla (\psi^d-\psi)+ \nabla{\psi} \cdot\nabla (\dot\psi^d-\dot\psi)\right)\dx,
\end{align*}
and we can rely on the following identities to deal with the nonlinear terms in \eqref{eq:adj1dereq}:
\begin{align*}
\pdfrac{}{t}\left((1-2k {\dot\psi}) (\dot\psi^d-\dot\psi)\right) =  (1-2k\dot\psi^d)\ddot \psi^d - (1-2k\dot\psi)\ddot \psi + k \pdfrac{}{t}(\dot\psi^d -\dot\psi)^2, 
\end{align*}
\begin{align*}
\nabla{\dot\psi} \cdot\nabla (\psi^d-\psi)+ \nabla{\psi} \cdot\nabla (\dot\psi^d-\dot\psi) =  \nabla \dot\psi^d \cdot \nabla \psi^d - \nabla \dot\psi \cdot \nabla \psi - \frac{1}{2} \pdfrac{}{t}\vert\nabla(\psi^d - \psi)\vert^2.
\end{align*}
Using these identities we infer
\begin{equation*}
\begin{aligned}
&\int_0^T \int_{\Omega}   j'(\psi) (\psi^d-\psi) \dx\dt\\
=&\, \begin{multlined}[t]  \int_0^T\int_{\Omega} \left\{(1-I_d) (1-2k\dot\psi^d)\ddot \psi^d p - k (\dot\psi^d -\dot\psi)^2 \dot p - \sigma \vert\nabla(\psi^d - \psi)\vert^2\dot p \right.  \\ 
 \left. +  c^2 (I-M_d)\nabla p \cdot\nabla \psi^d + b (I-M_d)\nabla p \cdot\nabla \dot\psi^d - 2 \sigma (I-M_d)(\nabla \dot\psi^d \cdot \nabla \psi^d) p  \right\}\dx \dt\\
 + \int_0^T\int_{\Omega} \left\{  I_d (1-2k\dot\psi^d)\ddot \psi^d p \right. \\+\left. \left(c^2 M_d \nabla p \cdot\nabla \psi^d + b M_d \nabla p \cdot\nabla \dot\psi^d - 2 \sigma (M_d \nabla \dot\psi^d \cdot \nabla \psi^d) p \right)\right\}\dx\dt \\
 - \int_0^T \int_{\Omega} \left\{  (1-2k\dot\psi)\ddot \psi p  -  \left(c^2 \nabla p \cdot\nabla \psi + b\nabla p \cdot\nabla \dot\psi - 2 \sigma (\nabla \dot\psi \cdot \nabla \psi) p \right)\right\} \dx\dt.\end{multlined}
\end{aligned}
\end{equation*}
If we use the weak form \eqref{eq:holder2} satisfied by $\psi^d$ with $\phi=p$ as the test function to replace the last three lines and further employ \cref{th:Holder} for the $\frac{1}{2}$-H\"{o}lder continuity of $\dot\psi$ and $\nabla\psi$, 
we obtain 
\begin{align*}
\begin{split}
&\lim_{d\rightarrow0} \frac{1}{d} \int_0^T \int_{\Omega} j'(\psi) (\psi^d-\psi) \dx \dt
\\=&\,\begin{multlined}[t] - \int_0^T\int_{\Omega} \left(\div h (1-2k\dot\psi)\ddot \psi p \right)\dx\dt  \\ 
  - \int_0^T\int_{\Omega} \left(c^2 M\nabla p \cdot\nabla \psi + b M\nabla p \cdot\nabla \dot\psi - 2 \sigma (M\nabla \dot\psi \cdot \nabla \psi) p \right) \dx\dt \\
 + \lim_{d\rightarrow0} \frac{1}{d}\left\{\int_0^T \int_{\bndry} w_d \left(c^2 g^d + b \dot g^d\right) p \, \ds\dt - \int_0^T \int_{\bndry} \left(c^2 g + b \dot g\right) p \, \ds\dt\right\}.\end{multlined}
\\
=&\, I+II+III,
\end{split}
\end{align*}
where $M$ is defined in \eqref{eq:smoothdefM}. Note that this passage to the limit is made possible by the uniform boundedness in $d$ of $\psi^d$ and $\dot{\psi^d}$ given by \cref{prop:uniformwellp}. 
This expression would correspond to the volume representation of the shape derivative. To arrive at the boundary representation predicted by the Delfour--Hadamard--Zol\'esio structure theorem, we integrate by parts the first term with respect to space
\begin{align*}
\begin{split}
I= & \int_0^T\int_{\Omega} (1-2k \dot \psi)\ddot \psi\nabla p \cdot h \dx\dt+ \int_0^T\int_{\Omega}  p\nabla((1-2k \dot \psi)\ddot \psi)\cdot  h  \dx\dt \\
& - \int_0^T\int_{\bndry} (1-2k\dot\psi)\ddot \psi p (h\cdot n)\dx\dt.
\end{split}
\end{align*}
We note that
\[
\nabla \left((1-2k\dot\psi)\ddot \psi \right) = -2k\ddot \psi\nabla \dot\psi + (1-2k\dot\psi) \nabla\ddot \psi = \pdfrac{}{t} \left((1-2k\dot\psi) \nabla \dot\psi\right).
\]
We then have after integration by parts ($h$ invariant with time):
\begin{align*}
\begin{split}
\int_0^T\int_{\Omega}  p\nabla((1-2k \dot \psi)\ddot \psi)\cdot  h  \dx\dt = & \int_0^T\int_{\Omega}  \pdfrac{}{t}\left((1-2k \dot \psi)\dot p\right) \nabla \psi\cdot  h  \dx\dt, 
\end{split}
\end{align*}
which yields 
\begin{align*}
\begin{split}
I = & \int_0^T\int_{\Omega} (1-2k \dot \psi)\ddot \psi\nabla p \cdot h \dx\dt+ \int_0^T\int_{\Omega}  \pdfrac{}{t}\left((1-2k \dot \psi)\dot p\right) \nabla \psi\cdot  h \dx\dt 
\\ &- \int_0^T\int_{\bndry} (1-2k\dot\psi)\ddot \psi p (h\cdot n) \dx\dt.
\end{split}
\end{align*}
Next we wish to transform $II$. By applying \cref{lem:trans2} to the term \(\int_{\Omega} (M\nabla \dot\psi \cdot \nabla \psi) p \dx,\) and integrating by part in time, we find that
\begin{align*}
\begin{split}
 &\int_0^T\int_{\Omega}(M\nabla  \dot\psi \cdot \nabla \psi) p \dx\dt\\
=& - \int_0^T\int_{\Omega} \left(\nabla \cdot (\dot p\nabla \psi) (h\cdot \nabla \psi) 
+ (\nabla \psi \cdot \nabla \dot\psi) (h\cdot\nabla p) \right)\dx\dt\\
& +\int_0^T\int_{\bndry} \dot p \pdfrac{\psi}{n}(h\cdot \nabla \psi)\ds\dt + \int_0^T\int_{\bndry} p\nabla \psi \cdot \nabla \dot\psi (h\cdot n)\ds\dt,
\end{split}
\end{align*}
Similarly, we use \cref{lem:trans2} to infer
\begin{align*}
\begin{split}
&- \int_0^T\int_{\Omega}  (c^2  M\nabla  p \cdot\nabla \psi +  b M\nabla p \cdot\nabla \dot\psi) \dx\dt\\
= &  - \int_0^T\int_{\Omega} (c^2 \laplace \psi  + b \laplace \dot \psi)(h\cdot \nabla p) \dx\dt
- \int_0^T\int_{\Omega} (c^2 \laplace p  - b \laplace  \dot p)(h\cdot \nabla \psi)\dx\dt\\
& + \int_0^T\int_{\bndry} (c^2\pdfrac{\psi}{n} + b\pdfrac{\dot \psi}{n})(h\cdot \nabla p) \ds\dt
+ \int_0^T\int_{\bndry} (c^2\pdfrac{p}{n} - b\pdfrac{\dot p}{n})(h\cdot \nabla \psi) \ds\dt \\
&-  \int_0^T\int_{\bndry} (c^2 \nabla p \cdot \nabla \psi + b \nabla p \cdot \nabla \dot \psi) (h\cdot n)\ds\dt.
\end{split}
\end{align*}
Altogether and using the Neumann data for $\psi$, this gives us that
\begin{align*}
\begin{split}
II =&  - \int_0^T\int_{\Omega} (c^2 \laplace \psi  + b \laplace \dot \psi+2 \sigma \nabla \psi \cdot \nabla \dot\psi)(h\cdot \nabla p) \dx\dt
\\ & - \int_0^T\int_{\Omega} (c^2 \laplace p  - b \laplace  \dot p+ 2 \sigma \nabla \cdot (\dot p\nabla \psi))(h\cdot \nabla \psi ) \dx\dt
\\&+ \int_0^T\int_{\bndry} (c^2g + b\dot g)(h\cdot \nabla p)\ds\dt \\ 
&-\int_0^T\int_{\bndry} (c^2 \nabla p \cdot \nabla \psi + b \nabla p \cdot \nabla \dot \psi 
- 2 \sigma p\nabla \psi \cdot \nabla \dot\psi ) (h\cdot n)\ds\dt.
\end{split}
\end{align*}
We use the rules of differentiation of mapped functions (see Appendix~\ref{ap:defprop}) to find 
\begin{align*}
\begin{split}
III = &\lim_{d\rightarrow0} \frac{1}{d}\left\{\int_0^T \int_{\bndry} w_d \left(c^2 g^d + b \dot g^d\right) p \ds\dt  - \int_0^T \int_{\bndry} \left(c^2 g + b \dot g\right) p \ds\dt \right\} \\
=& \Mostafa{\left(\int_0^T\int_{\bndry_d} \left(c^2 g^d + b \dot g^d\right) p\circ F_d^{-1} \ds\dt\right)'(0)}\\
=& \int_0^T \int_{\bndry} \left(c^2 g + b \dot g\right) (-\nabla p \cdot h) \ds\dt
\\&+ \int_0^T \int_{\bndry}(\pdfrac{}{n} ((c^2 g+b \dot g) p) + (c^2 g + b \dot g)p \kappa) (h\cdot n)\ds\dt. 
\end{split}
\end{align*}
Finally, calculating $I+II+III$, we conclude that
\begin{align*}
\begin{split}
&\diff J(\Omega) h =
\\& \int_0^T\int_{\Omega} \left((1-2k \dot \psi)\ddot \psi - c^2 \laplace \psi  - b \laplace \dot \psi-2 \sigma \nabla \psi \cdot \nabla \dot\psi\right)\nabla p \cdot h \dx\dt \\
& + \int_0^T\int_{\Omega} \left(\pdfrac{}{t}\left((1-2k \dot \psi)\dot p\right) - c^2 \laplace p  + b \laplace  \dot p- 2 \sigma \nabla \cdot (\dot p\nabla \psi) \right)\nabla \psi \cdot h \dx\dt\\
& + \int_0^T \int_{\bndry}(\pdfrac{}{n} ((c^2 g+b \dot g) p) + (c^2 g + b \dot g)p \kappa) (h\cdot n)  \ds\dt\\
&- \int_0^T\int_{\bndry} \left((1-2k\dot\psi)\ddot \psi p + (c^2 \nabla p \cdot \nabla \psi + b \nabla p \cdot \nabla \dot \psi - 2 \sigma p\nabla \psi \cdot \nabla \dot\psi )\right) (h\cdot n)\ds\dt.
\end{split}
\end{align*}
The second line above vanishes due to the fact that $\psi$ solves the wave equation \eqref{eq:Kuzn_potential}. The third line vanishes by using \eqref{eq:sysAdjKuz1} and then the fact that $\supp h \cap \overline {D_S} =\emptyset$.  Thus we arrive at \eqref{eq:diffexpr}, as claimed. \\
\indent Since $\Omega$ is $C^{2,1}$, the mean curvature $\kappa$ is well defined almost everywhere and belongs to $L^\infty(\bndry)$ \cite[p. 165]{abda2013dirichlet}. Note that due to the established regularities of $\psi$ in \cref{th:nonlinreg} and $p$ in \cref{th:adjreg}, the Eulerian shape derivative  is well defined.
Moreover $\diff J(\Omega) h$ is linear and continuous with respect to $h$. $J$ is, therefore, shape differentiable.
\end{proof}
\subsection{Other relevant objectives}\label{sec:finalTobj}
In what follows, we discuss how the previous analysis can be easily adapted to accommodate other practically relevant objectives. \\[2mm]
\textbf{Tracking the potential at final time.} In practice, we might want to
impose a desired ultrasound output at a target time-stamp \(t_T\in[0,T]\), usually final time $T$. In such cases, the cost function has the form
\begin{align}\label{eq:costf3}
	J_T(\Omega) = \frac{1}{2}\int_{\Omega} \left(\psi(T) - \psi_{D_S}\right)^2 \chi_{D_S} \dx= \int_{\Omega} j_T(\psi) \dx,
\end{align}
where we recall that $D_S \subset \overline{\Omega}$ is the ultrasound focal region. Similarly to before, we can assume $\supp h \cap \overline {D_S} =\emptyset$.

\indent With cost functional $J_T$, it can be shown that the strong form of the adjoint problem is formally given by
\begin{empheq}{align} \label{pb:adj3strong}
	\begin{array}{rl}
	\pdfrac{}{t}\left((1-2k{\dot\psi})\dot p\right) - c^2 \laplace p 
			 + b\laplace \dot p - 2\sigma \nabla \cdot\left(\dot p \nabla{\psi}\right)= 0& \qquad \mathrm{on\;} Q,\\
	c^2\pdfrac{p}{n}-b\pdfrac{\dot p}{n}+ 2\sigma g\dot p =0 &\qquad \mathrm{on\;} \Sigma,\\
	p(T) = 0&\qquad \mathrm{on\; \Omega},\\
	\dot p(T) =- \frac{\left(\psi(T)- \psi_{D_S} \right) \chi_{D_S}}{1-2k{\dot\psi}(T)}&\qquad \mathrm{on\; \Omega}.
	\end{array}
\end{empheq}
This initial-boundary value problem fits into the general theoretical framework of \cref{th:adjreg}. Since the new source term is conveniently zero, all that needs to be shown is that 
$$\dot p(T) =- \frac{\left(\psi(T)- \psi_{D_S} \right) \chi_{D_S}}{1-2k{\dot\psi}(T)} \in H^1(\Omega).$$
This is the case if $\psi_{D_S} \in H^1(\Omega)$. 
We can see that $\dot p(T) \in L^2(\Omega)$. Writing the gradient of $\dot p(T)$ and noticing that $1-2k{\dot\psi}(T) \in L^\infty(\Omega)$ is bounded away from zero ($\psi \in \mathcal{W}$ defined in \eqref{eq:Wballdef}), 
one can readily prove the desired regularity. And subsequently, due to \cref{th:adjreg}, initial boundary value problem \eqref{pb:adj3strong} has a unique solution $p$ in the space $\mathcal{P}$, which is defined in \eqref{eq:Pdef}. 

Using the same rearrangement technique as before, we write the difference and using \cref{th:Holder} (with the embedding $H^1(H^1) \hookrightarrow C^0(H^1)$) together with the  differentiability of $\psi_{D_S}^d$ at $d = 0$ 
and the fact that $\supp h \cap \overline{D_S} = \emptyset$, we infer that 
\begin{align*}
\diff J_T(\Omega) h = & \lim_{d\rightarrow0} \frac{1}{d} \int_{\Omega} j_T'(\psi) (\psi^d(T)-\psi(T)) \dx.
\end{align*}
Testing the weak form \eqref{pb:adj3strong} of the adjoint problem with $\psi^d-\psi \in H^1(0,T; H^1(\Omega))$, we obtain
\begin{align}\label{adj3dereq}
\begin{split}
\int_{\Omega} j_T'(\psi) (\psi^d(T)-\psi(T)) \dx=  - \int_0^T\int_{\Omega} \left((1-2k {\dot\psi})\dot p\right) (\dot \psi^d - \psid) \dxs
\\ +\int_0^T\int_{\Omega}(c^2 \nabla p - b\nabla\dot p + \sigma \dot p \nabla{\psi}) \cdot\nabla(\psi^d-\psi)\dxs.
\end{split}
\end{align}
We further use the fact that $p(T) = \psid^d (0) = \psid (0) = 0$ to integrate once in time the first term 
\begin{multline*}
- \int_0^T\int_{\Omega}(1-2k {\dot\psi})\dot p (\dot \psi^d - \psid) \dxs=\hspace{-3pt} \int_0^T\int_{\Omega} \pdfrac{}{t}\left((1-2k {\dot\psi})(\dot \psi^d - \psid)\right) p \dxs.
\end{multline*}
With this, the right-hand side of \eqref{adj3dereq} is the same as the one obtained in \eqref{eq:adj1dereq}. From there we can continue the computations done in \cref{th:derth1}. We find that the shape derivative has again the form \eqref{eq:diffexpr}.\\[2mm]
\noindent \textbf{Tracking the pressure.} In ultrasound applications, it might be desirable to impose that the pressure 
match a desired value on $D$; i.e,
\begin{align*}
\textup{p} = \textup{p}_D \quad \mathrm{on\;D} \iff \rho\dot \psi = \textup{p}_{D}& \iff \dot \psi = \frac{\textup{p}_{D}}{\rho} \\
& \iff \dot \psi = f_D, \quad \mathrm{where\;}f_D = \frac{\textup{p}_{D}}{\rho}.
\end{align*}
Here $\rho$ is the mass density at each point of the medium on the 
horizon $[0,T]$ or, alternatively, \(\rho : [0,T] \times \overline{\Omega} \rightarrow \Rs_{>0}\). We assume $f_D \in H^1(0,T;L^2(\Omega)) \cap L^\infty(0,T;H^1(\Omega))$ and consider

\begin{align*}
	J_\textup{p}(\psi,\Omega) = \frac{1}{2}\int_0^T \int_{\Omega} \left(\dot\psi - f_D\right)^2 \chi_D \dxs = \int_0^T \int_{\Omega} j_\textup{p}(\psi)\dxs.
\end{align*}
\noindent For $J_\textup{p}$, the strong form of the adjoint problem is given by
\begin{empheq}{align}\label{eq:sysAdjKuz2}
	\begin{array}{rl}
	\pdfrac{}{t}\left((1-2k{\dot\psi})\dot p\right) - c^2 \laplace p 
			 + b\laplace \dot p - 2 \nabla \cdot\left(\dot p \nabla{\psi}\right)= - ({\ddot \psi} - \dot f_D)\chi_D& \ \mathrm{on\;} Q,\\
	c^2\pdfrac{p}{n}-b\pdfrac{\dot p}{n}+ 2g\dot p =0 &\ \mathrm{on\;} \Sigma,\\
	p(T) = 0&\ \mathrm{on\; \Omega},\\
	\dot p(T) =-\frac{\left(\dot\psi(T) - f_D(T)\right)\chi_{D_S}}{1-2k\dot \psi(T)}&\ \mathrm{on\; \Omega}.
	\end{array}
\end{empheq}
One can readily show that \eqref{eq:sysAdjKuz2} verifies the assumptions of \cref{th:adjreg}.
Similarly to before, we compare the functional $J_\textup{p}$ on a reference and a deformed domain, rely on the H\"older continuity of $\psid$, 
together with the fact that $\psi_{D_S}^d=\psi_{D_S}^d(d)$ is differentiable at $d = 0$ to infer that
\begin{align*}
\diff J_\textup{p}(\Omega) h = & \lim_{d\rightarrow0} \frac{1}{d} \int_0^T\int_{\Omega} j_\textup{p}'(\psi) (\psid^d-\psid) \dx\dt.
\end{align*}
We test \eqref{eq:sysAdjKuz2} by $\psi^d - \psi \in H^1(0,T;H^1(\Omega))$ and use the fact that $p(T) = \psid^d (0) = \psid (0) = 0$ to integrate once in time the first term to get
\begin{align}\label{adj2dereq}
\begin{split}
\int_0^T\int_{\Omega} j_\textup{p}'(\psi) (\psid^d-\psid) \dx\dt =  \int_0^T\int_{\Omega} \pdfrac{}{t}\left((1-2k {\dot\psi})(\dot \psi^d - \psid)\right) p \dxs  
\\ + \int_0^T\int_{\Omega}(c^2 \nabla p - b\nabla\dot p  + 2 \sigma \dot p \nabla{\psi}) \cdot\nabla (\psi^d-\psi)\dxs.
\end{split}
\end{align}
The right-hand side of \eqref{adj2dereq} is the same as the one obtained in \eqref{eq:adj1dereq}. From there, we can continue the computations as done in \cref{th:derth1}. We find, similarly, that the shape derivative is of the form \eqref{eq:diffexpr}. The established regularities of $\psi$ in \cref{th:nonlinreg} and $p$ in \cref{th:adjreg}, imply that the shape derivative is well defined.

\section*{Conclusion and outlook}
In this work, we have analyzed shape optimization problems governed by general wave equations that model nonlinear ultrasound propagation and, as such, arise in HIFU applications. In particular, we have established sufficient conditions for the well-posedness and regularity of the underlying wave models with nonhomogeneous Neumann boundary conditions, uniformly with respect to shape deformations, as well as the H\"older continuity of the solutions. Furthermore, we have studied the corresponding adjoint problems and rigorously computed shape derivatives for several objectives of practical interest.

Our results provide a sound basis for the analysis and implementation of suitable numerical algorithms for shape optimization of HIFU waves.
\Mostafa{Indeed, \cref{th:derth1} provides a derivative expression which can be used in a gradient-descent algorithm to find, e.g., the optimal arrangement of piezoelectric trandsucers in a HIFU device.
Calculating the shape derivative necessitates the solving of only two PDEs (a nonlinear state problem and a linear adjoint problem). A simple integral formula gives then the derivative for any vector field. This makes designing HIFU devices using the established formulas for shape sensitivities particularly attractive.}
  Future work will also be concerned with generalizing the presented theoretical framework to allow for sound propagation through media with different relaxation mechanisms.

\appendix
\section{Properties of the perturbation of identity mapping}\label{ap:defprop}
In this appendix, we collect certain helpful properties of the perturbation of identity mapping. We refer to \cite[Lemmas 2.4, 4.2, 4.3 \& 4.8]{murat1976controle}, \cite[Section 2.3.2]{delfour2011shapes}, and
 \cite[Proposition 6.6.2]{berger2012differential} for the following results. \\
 \indent Let $\I = [0,\delta_0]$ with $\delta_0>0$ sufficiently small. Then the following properties hold:
\begin{align*} 
\begin{array}{ll}
F_0 = id,  & I_0 = 1, \\ 
d\rightarrow F_d \in C(\I, C^{2,1} (\overline U, \Rs^l)), &d\rightarrow F_d^{-1} \in C(\I, C^{2,1} (\overline U, \Rs^l)), \\[1mm]
d\rightarrow A_d \in C^1(\I, C^{1,1}(\overline U, \Rs^{l\times l})), &d\rightarrow I_d \in C^1(\I, C^{1,1}(\overline U)), \\[1mm]
& d\rightarrow w_d \in C(\I, C^{1,1}(\bndry)) \cap C^1 (\I, C(\bndry)),   \\
\\
\Mostafa{\left(F_d\right)'(0)}= h, & \Mostafa{\left(F_d^{-1}\right)'(0)} = -h \\[1mm]
\Mostafa{\left(DF_d\right)'(0)} = \nabla h & \Mostafa{\left(DF_d^{-1}\right)'(0)} = \Mostafa{\left(A_d^T\right)'(0)} = -\nabla h \\[1mm]
\Mostafa{\left(I_d\right)'(0)} = \div h & \Mostafa{\left(w_d\right)'(0)} = \div_{\bndry}h = \left. \div h \right\vert_{\bndry} - \nabla h n\cdot n,
\end{array}
\end{align*}
\Mostafa{where, again, $(\cdot)'$ stands for the derivative the deformation mappings with respect to $d$.}

We also recall here some rules of differentiation and integration of the mapped functions; see, for example, 
Propositions 2.29, 2.47, and 2.50 in \cite[Chapter 2]{sokolowski1992introduction}, and Theorem 4.3 in \cite[Chapter 9]{delfour2011shapes} for their proofs. \vspace*{2mm}

\begin{itemize}
  \item Let $\varphi_d \in L^1(\Omega_d)$, then $\varphi_d \circ F_d \in L^1(\Omega_d)$: 
  \[
  \int_{\Omega_d}\varphi_d \dx_d = \int_{\Omega}\varphi_d \circ F_d \det DF_d \dx = \int_{\Omega}I_d\varphi^d  \dx
  \]
  \item Let $\varphi_d \in L^1(\bndry_d)$, then $\varphi^d \in L^1(\bndry)$: 
  \[
  \int_{\bndry_d}\varphi_d \dg_d = \int_{\bndry} w_d \varphi^d \dg,
  \]
  \item $\phi_d \in H^1(\Omega_d)$ if and only if $\phi^d \in H^1(\Omega)$: 
  \[
  (\nabla \varphi_d) \circ F_d = A_d \nabla \varphi^d.
  \] 
  \item Assume that $f\in C((-\delta_0,\delta_0), W^{2,1}(U))$ and $f_d(0)$ exists in $W^{1,1}(U)$. Then
  \begin{align*}
  &\Mostafa{\left( \int_{\bndry_d} f(d,\gamma) \dg \right)'(0)} = \\
  &\int_{\bndry} \left\{ \Mostafa{\left(  f(d,\gamma) \right)'(0)} + \left(\pdfrac{}{n} f(0,\gamma) + \kappa f(0,\gamma)\right)
  (h\cdot n) \right\} \dg,
 \end{align*} 
\end{itemize}
where $\kappa$ stands for the mean curvature of $\bndry$. 

\bibliographystyle{abbrv}
\bibliography{references}
\end{document}